\documentclass[psamsfonts]{amsart}

\usepackage{amssymb,amsfonts}
\usepackage{tikz-cd} 
\usepackage[all,arc]{xy}
\usepackage{enumerate}
\usepackage{mathrsfs}
\usepackage{hyperref}
\usepackage{relsize}
\hypersetup{
pdftitle={Generalised Temperley-Lieb algebras of type $G(r,p,n)$},
pdfsubject={Mathematics},
pdfauthor={Gus Lehrer AND Mengfan Lyu},
pdfkeywords={Temperley-Lieb algebra, Hecke algebras, cellular basis, decomposition numbers}}
\usepackage{bm}
\usepackage{pgfplots}

\usepackage{cite}

\newcommand{\mla}{\mathlarger}

\newtheorem{thm}{Theorem}[section]
\newtheorem{cor}[thm]{Corollary}
\newtheorem{prop}[thm]{Proposition}
\newtheorem{lem}[thm]{Lemma}

\theoremstyle{definition}
\newtheorem{defn}[thm]{Definition}
\newtheorem{exmp}[thm]{Example}

\theoremstyle{remark}
\newtheorem{rem}[thm]{Remark}

\makeatletter
\let\c@equation\c@thm
\makeatother
\numberwithin{equation}{section}

\title{Generalised Temperley-Lieb algebras of type $G(r,p,n)$}

\author{Gus Lehrer and Mengfan Lyu}
\address{School of Mathematics and Statistics,
University of Sydney, N.S.W. 2006, Australia}
\email{gustav.lehrer@sydney.edu.au,mengfan.lyu@uts.edu.au }

\thanks{The authors wish to thank the referee for their very comprehensive reading of the manuscript, and in particular for pointing out some errors,
and for numerous suggestions which have significantly improved the exposition.}

\makeatletter
\@namedef{subjclassname@2020}{\textup{2020} Mathematics Subject Classification}
\makeatother

\begin{document}

\begin{abstract} In an earlier work, we defined a ``generalised Temperley-Lieb algebra'' $TL_{r,1,n}$ corresponding to the imprimitive reflection group $G(r,1,n)$ as a quotient of the cyclotomic Hecke algebra. In this work we introduce the generalised Temperley-Lieb algebra $TL_{r,p,n}$ which corresponds to the complex reflection group $G(r,p,n)$. Our definition identifies $TL_{r,p,n}$ as the fixed-point subalgebra of $TL_{r,1,n}$ under a certain automorphism $\sigma$. We prove the cellularity of $TL_{r,p,n}$ by proving that $\sigma$ induces a special shift automorphism with respect to the cellular structure of $TL_{r,1,n}$. We also give a description of the cell modules of $TL_{r,p,n}$ and their decomposition numbers, and finally we point to how our algebras might be categorified and could lead to a diagrammatic theory.
\end{abstract}
%\\\\
%\\\\
%\noindent \textbf{2010 Math. Subject Class.} 46G05, 30G25, 39-XX\\
%\noindent \textbf{Personal website:} \ \href{https://goo.gl/P9Ttmu}{kolosovpetro.github.io}\\
%\noindent \textbf{ORCID:} \ \href{https://goo.gl/qqxSYM}{0000-0002-6544-8880}\\
%\noindent \textbf{e-mail:} \ \ \ \href{mailto:kolosovp94@gmail.com}{kolosovp94@gmail.com}\\

\maketitle
\tableofcontents

\section{Introduction}

The classical Temperley-Lieb algebra was initially introduced in \cite{TL1971} in connection with transition matrices in statistical mechanics. 
Since then, it has been linked with many diverse areas of mathematics, such as operator algebras, quantum groups, categorification and representation theory.
 It can be defined in several ways, including as a quotient of the Hecke algebra of type $A$, or as an associative diagram algebra, i.e. an algebra whose elements are linear combinations of diagrams
 with certain composition rules.

In \cite{LL22},  we defined a generalised Temperley-Lieb algebra $TL_{r,1,n}$ corresponding to the complex reflection group $G(r,1,n)$ as a generalisation of the Temperley-Lieb algebras of types $A_{n-1}$ and $B_n$. 
In that work, we give a cellular structure for $TL_{r,1,n}$ and find the decomposition matrix of the corresponding cell modules using the technology of KLR algebras.

We generalise that work in this paper. Specifically, we introduce a generalised Temperley-Lieb algebra $TL_{r,p,n}$ which corresponds to an arbitrary imprimitive complex reflection group $G(r,p,n)$. It is defined as the fixed-point subalgebra of a suitably specialised generalised Temperley-Lieb algebra $TL_{r,1,n}$ under an automorphism $\sigma_{TL}$ defined in (\ref{eq817}).

Moreover, we introduce a modified multipartition which is called a 3-dimensional multipartition in Section \ref{3dmlp}, and construct a new cellular structure for our specialised $TL_{r,1,n}$ in Theorem \ref{8.2.5}. Inspired by the skew cellularity introduced by Hu, Mathas and Rostam in \cite{hu2021skew}, we show that $\sigma_{TL}$ induces a shift automorphism with respect to this cellular basis in Theorem \ref{thshift} and deduce the cellularity of $TL_{r,p,n}$ in Theorem \ref{csrpn}.
Using this cellular structure of $TL_{r,p,n}$, which is inherited from $TL_{r,1,n}$, we determine the decomposition matrix of the cell modules of $TL_{r,p,n}$ in Theorem \ref{finaldcm}.

In \S9, we indicate how our algebras might be categorified, and provide some hints for their realisation as diagram algebras.

\section{The cyclotomic Hecke algebra $H(r,p,n)$} \label{8.1.1}
To define the generalised Temperley-Lieb algebra $TL_{r,p,n}$, we first recall the definition of the cyclotomic Hecke algebra $H(r,p,n)$ corresponding to the imprimitive reflection group $G(r,p,n)$. In this section, we also recall Ariki's construction in \cite{ARIKI1995164} showing that $H(r,p,n)$ is the fixed point subalgebra under an automorphism $\sigma$ of $H_n^{\Lambda}(q,\zeta)$, a specialisation of $H(r,1,n)$. We will use this automorphism to define $TL_{r,p,n}$.

The following lemma gives a presentation of the imprimitive reflection group $G(r,p,n)$.
\begin{lem} (cf. \cite{broue1994complex},Appendix 2)
	Let $r,p$ and $n$ be positive integers such that $p|r$ and let $d=\dfrac{r}{p}$. The complex reflection group $G(r,p,n)$ is the group generated by $s_0,s_1,s_1',s_2,$ $s_3,\dots,s_{n-1}$ subject to the following relations:
	\begin{align*}
		s_0^d=s_1^2=s_1'^2=s_2^2&=\dots=s_{n-1}^2=1;\\
		s_0s_1s_0s_1&=s_1s_0s_1s_0; \\
		s_0s_1's_0s_1'&=s_1's_0s_1's_0;\\
		s_0s_1s_1'&=s_1s_1's_0;\\
		s_0s_i=s_is_0&\text{ for }i\geq 2;\\
		s_1's_2s_1'&=s_2s_1's_2;\\
		s_1's_i=s_is_1'&\text{ for } i\geq 3;\\
		s_is_{i+1}s_i=s_{i+1}s_is_{i+1} & \text{ for }i\geq 1;\\
		s_is_j=s_js_i & \text{ for }|i-j|\geq 2 \text{ and }i,j\geq 1;\\
		 \underbrace{s_0s_1s_1's_1s_1'\dots}_{p+1} &= \underbrace{s_1's_0s_1s_1's_1\dots}_{p+1}.
	\end{align*}
\end{lem}

$G(r,p,n)$ may be viewed as the group consisting of all $n\times n$ monomial matrices such that each non-zero entry is an $r^{th}$ root of unity and the product of
the non-zero entries is a $d^{th}$ root of unity, where
 $d=\frac{r}{p}$. More precisely, let $\zeta \in \mathbb{C}^*$ be a primitive $r^{th}$ root of unity. The complex reflection group $G(r,p,n)$ is generated by the following elements:
\begin{align*}
	s_0&=\zeta^pE_{1,1}+\sum_{k=2}^{n}E_{k,k},\\
	s_1'&=\zeta E_{1,2}+\zeta^{-1}E_{2,1}+\sum_{k=3}^{n}E_{k,k}\\
	s_i&=\sum_{1\leq k\leq i-1}E_{k,k}+E_{i+1,i}+E_{i,i+1}+\sum_{i+2\leq k\leq n}E_{k,k} \text{ for all }1\leq i\leq n-1
\end{align*}
where $E_{i,j}$ is the elementary $n\times n$ matrix with $(i,j)$-entry equal to $1$ and all other entries equal to zero. Since the complex reflection group $G(r,1,n)$ 
is generated by $s_i(1\leq i\leq n-1)$ as well as the following matrix:
\begin{equation*}
    t_0=\zeta E_{1,1}+\sum_{k=2}^{n}E_{k,k},
\end{equation*}
the complex reflection group $G(r,p,n)$ is a normal subgroup of $G(r,1,n)$.

The Hecke algebra of type $G(r,p,n)$ is a deformation of its complex group ring. Following \cite{broue1994complex}, the Hecke algebra $H(r,p,n)$ may be defined as follows:
\begin{defn} (\cite{broue1994complex},Definition 4.21)\label{dfhrpn}
	Let $R$ be a commutative ring with 1 and $q,u_1,u_2,\dots,u_r\in R^*$.
	The Hecke algebra $H(r,p,n)$ is the unitary associative $R$-algebra generated by $S,T_1',T_1,T_2,\dots,T_{n-1}$ subject to the following relations:
		\begin{align*}
		(S-u_1)(S-u_2)\dots&(S-u_d)=0;\\
		(T_1'-q)(T_1'+1)=(T_i-q)&(T_i+1)=0 \text{  for }1\leq i\leq n-1;\\
		ST_1ST_1&=T_1ST_1S;\\
		ST_1'ST_1'&=T_1'ST_1'S;\\
		ST_1T_1'&=T_1T_1'S;\\
		ST_i=T_iS&\text{ for }i\geq 2;\\
		T_1'T_2T_1'&=T_2T_1'T_2;\\
		T_1'T_i=T_iT_1'&\text{ for } i\geq 3;\\
		T_iT_{i+1}T_i=T_{i+1}T_iT_{i+1} & \text{ for }i\geq 1;\\
		T_iT_j=T_jT_i  \text{ for }&|i-j|\geq 2 \text{ and }i,j\geq 1;\\
		\underbrace{ST_1T_1'T_1T_1'\dots}_{p+1} &= \underbrace{T_1'ST_1T_1'T_1\dots}_{p+1}.
	\end{align*}
\end{defn}
In analogy with the restrictions on the parameters in \cite{LL22}, we require that $R$ be a field of characteristic 0 with a $p^{th}$ primitive root of unity and $q,u_1,u_2,\dots,u_d\in R^*$ such that 
\begin{equation}\label{uures}
    \frac{u_i}{u_j}\neq 1,q\textbf{ or } q^2
\end{equation}
for all $i\neq j$ and
\begin{equation}
    (1+q)(1+q+q^2)\neq 0.
\end{equation}
Moreover, we assume that there exists $w_i\in R$ such that $w_i^p=u_i$ for $i=1,2,\dots,d$. 

We see from their matrix realisations that $G(r,p,n)$ is a subgroup of $G(r,1,n)$. It is a natural question to ask whether there is a connection between the two Hecke algebras.
To review the construction of the Hecke algebra of $G(r,p,n)$ by Ariki in \cite{ARIKI1995164} which indicates the connection between $H(r,1,n)$ and $H(r,p,n)$, we recall the definition of $H(r,1,n)$.

%Let  $v_1,v_2,\dots,v_r\in R$ be indeterminates over $\mathbb{C}$. The cyclotomic Hecke algebra corresponding to $G(r,1,n)$ over $R$ is defined as follows:
\begin{defn}\label{Hr}(\cite{broue1994complex},Definition 4.21)
	Let $R$ be as defined above and $q,v_1,v_2,\dots,v_r\in R^*$. The cyclotomic Hecke algebra $H(r,1,n)$ corresponding to $G(r,1,n)$ over $R$ is the associative algebra 
	generated by $T_0,T_1,\dots,T_{n-1}$ subject to the  following relations:
	\begin{align}
		(T_0-v_1)(T_0-v_2)\dots (T_0-v_r)&=0 ;\label{ordering relation}\\
		(T_i-q)(T_i+1)&=0\textit{ for } 1\leq i\leq n-1;\label{quadra rlt}\\
		T_0T_1T_0T_1&=T_1T_0T_1T_0;\\
		T_iT_{i+1}T_i&=T_{i+1}T_iT_{i+1}\textit{ for }1\leq i\leq n-2;\\
		T_iT_j&=T_jT_i \textit{ for }|i-j|\geq 2.
		\end{align}
\end{defn} 
We remind readers that in the original definition of $H(r,1,n)$, the base field $R$ is only required as a commutative ring with unity. The restrictions on $R$ are for the following specialisation:

Let $\zeta\in R$ be a primitive root of unity of order $p$. Choose the parameters $v_1,v_2,\dots,v_r$ in $(\ref{ordering relation})$ so that the equation (2.6) takes  the following form:
\begin{equation} \label{eq8.3}
	\prod_{k=1}^{d}\prod_{i=1}^{p}(T_0-\zeta^i w_k)=0,
\end{equation}
where $w_k$ is an element in $R$ such that $w_k^p=u_k$.
This equation can be written as follows:
 \begin{equation}
 	\prod_{k=1}^{d}(T_0^p-u_k)=0.
 \end{equation}

Denote by $H_n(q,u_k)$ the algebra $H(r,1,n)$ specialised as above. Then we have an automorphism $\sigma:H_n(q,u_k) \mapsto H_n(q,u_k) $ such that:
\begin{equation}\label{eq:defsigma}
	\sigma(T_0)=\zeta T_0,\;\;\; \sigma(T_i)=T_i\text{ for all } 1\leq i\leq n-1.
\end{equation}

Ariki shows in \cite{ARIKI1995164}:
\begin{prop} \label{ariki1.6}
	(\cite{ARIKI1995164}, Proposition 1.6) The algebra homomorphism $\phi: H(r,p,n)\to H_n(q,u_i)$ given by:
	\begin{align*}
		\phi(S)&=T_0^p,\\
		\phi(T_1')&=T_0^{-1}T_1T_0,\\
		\phi(T_i)&=T_i \text{ for }1\leq i\leq n-1
	\end{align*}
    is well-defined and one-to-one. And $H(r,p,n)\cong \phi( H(r,p,n))$.
\end{prop}

Further, we have
\begin{prop} (\cite{Rostam_2018},Corollary 1.18) \label{rosprop}
	The algebra $H(r,p,n)$ is isomorphic to $H_n(q,u_k)^{\sigma}$ via $\phi$, where the latter algebra is the fixed point subalgebra of $H_n(q,u_k)$ under $\sigma$.
\end{prop}

This proposition implies that the cyclotomic Hecke algebra $H(r,p,n)$ can be realised as the subalgebra of a specialisation of the cyclotomic Hecke algebra $H(r,1,n)$, consisting of the points fixed by the automorphism $\sigma$ defined in (\ref{sigma}). In analogy with this situation, we will define the generalised Temperley-Lieb algebra $TL_{r,p,n}$ as a  fixed point subalgebra of $TL_{r,1,n}$ under $\sigma$ in Section \ref{subsec:813}.

In the remaining part of this section, we interpret the automorphism $\sigma$ in the language of quivers which are used to define a KLR algebra, 
from which we obtain the cellularity of $TL_{r,p,n}$. This interpretation was first discussed by Rostam in Section 3.1 in \cite{Rostam_2018}.

Let $q\in R^*$ be the parameter in definition \ref{dfhrpn} and $e$ be the smallest positive integer such that $1+ q + \dots+ q ^{e-1} = 0$, setting $e := 0$ 
if no such integer exists. Let $\Gamma_{e,p}$ be the quiver with vertex set $K:=I\times J$ where $I= \mathbb{Z}/p\mathbb{Z}$ and $J= \mathbb{Z}/e\mathbb{Z}$, 
and with a directed edge from $(i_1,j_1)$ to $(i_1,j_1+1)$. As there is no edge between $(i,j)$ and $(i',j')$ if $i\neq i'$ in the quiver $\Gamma_{e,p}$, 
it decomposes as $p$ layers which are isomorphic to each other. For example, $\Gamma_{0,3}$ consists of three layers, each of them is a quiver of type $A_{\infty}$,
as depicted below.

\begin{tikzcd}
  \Gamma_{0,3}   :   &  \dots    \arrow[r ] & (0,-1)\arrow[r ] &   (0,0)  \arrow[r ] &   (0,1)   \arrow[r ] &    (0,2) \arrow[r ] &    \dots\\
 &  \dots    \arrow[r ] & (1,-1)\arrow[r ] &   (1,0)  \arrow[r ] &   (1,1)   \arrow[r ] &    (1,2) \arrow[r ] &    \dots\\
 &  \dots    \arrow[r ] & (2,-1)\arrow[r ] &   (2,0)  \arrow[r ] &   (2,1)   \arrow[r ] &    (2,2) \arrow[r ] &    \dots
 \end{tikzcd}

Let $(c_{(i,j)(i',j')})$ be the Cartan matrix associated with $\Gamma_{e,p}$, so that
\begin{gather} \label{216}
	c_{(i,j)(i',j')}=
	\begin{cases}
		2 &\textit{ if } (i,j)=(i',j');\\
		-1 & \textit{ if } e\neq 2,i=i'\textit{ and } j=j'\pm1;\\
		-2 & \textit{ if } e= 2, i=i'\textit{ and } j=j'\pm1;\\
		0 &\textit{ otherwise.} 
	\end{cases}
\end{gather}
Let $\{\alpha_{(i,j)}|(i,j)\in K\}$ be the associated set of simple roots and $\{\Lambda_{(i,j)}|(i,j)\in K\}$ be the set of fundamental weights. Let $(-,-)$ be the bilinear form determined by
\begin{gather*}
	(\alpha_{(i,j)},\alpha_{(i',j')})=c_{(i,j)(i',j')} \textit{ and }(\Lambda_{(i,j)},\alpha_{(i',j')})=\delta_{(i,j)(i',j')}.
\end{gather*}

Let $P_+=\oplus_{(i,j)\in K}\mathbb{N}\Lambda_{(i,j)}$ be the dominant weight lattice and  $Q_+=\oplus_{(i,j)\in K}\mathbb{N}\alpha_{(i,j)}$ be the positive root lattice. For $\Lambda\in P_+$, define the 
length of $\Lambda$, $$l(\Lambda):=\sum_{(i,j)\in K}(\Lambda,\alpha_{(i,j)}),$$ and for $\alpha\in Q_+$, the height of $\alpha$, $ht(\alpha):=\sum_{(i,j)\in K}(\alpha,\Lambda_{(i,j)})$. 

Fix a dominant weight $\Lambda$ such that $l(\Lambda)=r$ and $(\Lambda,\alpha_{(i,j)})=(\Lambda,\alpha_{(i',j)})$ for all $i,i'\in I\text{ and }j\in J$.
Let $\zeta$ be a $p^{th}$ primitive root of unity and choose the parameters $v_1,v_2,\dots,v_r$ so that the equation $(\ref{ordering relation})$ becomes:
\begin{equation} \label{domwt}
	\prod_{i\in I}\prod_{j\in J}(T_0-\zeta^iq^j)^{(\Lambda,\alpha_{(i,j)})}=0,
\end{equation}
 which, due to our choice of $\Lambda$, can be written as
 \begin{equation}
 	\prod_{j\in J}(T_0^p-q^{pj})^{(\Lambda,\alpha_{(i,j)})}=0.
 \end{equation}

Denote by $H_n^{\Lambda}(q,\zeta)$ the algebra $H(r,1,n)$ specialised as above. Then $H_n^{\Lambda}(q,\zeta)$ can be also regarded as a specialisation of $H_n(q,u_k)$ which is given by $(\ref{eq8.3})$ with $w_k=q^{j_k}$ for some $j_k\in J$. Then we have an automorphism $\sigma:H_n^{\Lambda}(q,\zeta) \mapsto H_n^{\Lambda}(q,\zeta) $ such that:
\begin{equation} \label{sigma}
	\sigma(T_0)=\zeta T_0, \sigma(T_i)=T_i\text{ for all } 1\leq i\leq n-1.
\end{equation}

Similarly, choose the parameters $u_1,\dots,u_d$ such that the first equation in Definition \ref{dfhrpn} becomes
\begin{equation}
	\prod_{j\in J}(S-q^{pj})^{(\Lambda,\alpha_{(i,j)})}=0
\end{equation}
and denote by $H_{p,n}^{\Lambda}(q)$ the new algebra with this specialisation. 
As a direct consequence of Proposition \ref{rosprop}, the algebra $H_{p,n}^{\Lambda}(q)$ is isomorphic to $H_n^{\Lambda}(q,\zeta)^{\sigma}$ via the automorphism $\phi$ given in Proposition \ref{ariki1.6}.

We next recall the definition of the cyclotomic KLR algebra associated with the quiver $\Gamma_{e,p}$ and the weight $\Lambda$. The following definition is originally from the work of Khovanov and Lauda in \cite{khovanov2008diagrammatic} and Rouquier in \cite{rouquier20082kacmoody}. We shall use some of the notation introduced here in our exposition below.

\begin{defn} $\label{klrdf}$
	The cyclotomic KLR algebra associated with the weight $\Lambda$ is defined as $\mathcal{R}_n^{\Lambda}:=\oplus_{ht(\alpha)=n}\mathcal{R}_{\alpha}^{\Lambda}$, where $\mathcal{R}_{\alpha}^{\Lambda}$ is the unital associative $F$-algebra with generators
	\begin{equation}
		\{\psi_1,\dots,\psi_{n-1}\}\cup\{y_1,y_2,\dots,y_n\}\cup \{e(\mathbf{i})|\mathbf{i}\in K^{\alpha}\}
	\end{equation}
	subject to  the following relations
	\begin{equation}
		y_1^{(\Lambda,\alpha_{\mathbf{i}_1})}e(\mathbf{i})=0,
	\end{equation}
	\begin{equation}
		e(\mathbf{i})e(\mathbf{i}')=\delta_{\mathbf{i}\mathbf{i}'}e(\mathbf{i}), \label{224eq}
	\end{equation}
	\begin{equation}
		\sum_{\mathbf{i}\in K^{\alpha}}e(\mathbf{i})=1,
	\end{equation}
	\begin{equation} \label{18eq}
		e(\mathbf{i})y_r=y_r e(\mathbf{i}),
	\end{equation} 
	\begin{equation} 
		\psi_r e(\mathbf{i})=e(s_r\circ \mathbf{i})\psi_r \label{sym},
	\end{equation} 
	\begin{equation}
		y_r y_s=y_sy_r,
	\end{equation} 
	\begin{equation}\label{21eq}
		\psi_r y_s=y_s \psi_r \textit{ if }s\neq r,r+1;
	\end{equation}
	\begin{equation} \label{22eq}
		\psi_r \psi_s=\psi_s \psi_r \textit{ if }|r-s|\neq 1;
	\end{equation}
	\begin{gather} \label{23eq}
		\psi_r y_{r+1}e(\mathbf{i})=
		\begin{cases}
			y_r \psi_r e(\mathbf{i})+e(\mathbf{i}) &\textit{ if } \mathbf{i}_r= \mathbf{i}_{r+1};\\
			y_r \psi_r e(\mathbf{i}) &\textit{ if } \mathbf{i}_r\neq \mathbf{i}_{r+1};
		\end{cases}
	\end{gather}
	\begin{gather} \label{24eq}
		y_{r+1} \psi_r e(\mathbf{i})=
		\begin{cases}
			\psi_r y_r e(\mathbf{i})+e(\mathbf{i}) &\textit{ if } \mathbf{i}_r= \mathbf{i}_{r+1};\\
			\psi_r y_r e(\mathbf{i}) &\textit{ if } \mathbf{i}_r\neq \mathbf{i}_{r+1};
		\end{cases}
	\end{gather}
	\begin{gather} \label{25eq}
		\psi_r^2e(\mathbf{i})=
		\begin{cases}
			0 &\textit{ if } \mathbf{i}_r= \mathbf{i}_{r+1};\\
			e(\mathbf{i}) &\textit{ if } \mathbf{i}_r\neq \mathbf{i}_{r+1}\pm 1;\\
			(y_{r+1}-y_r)e(\mathbf{i}) &\textit{ if } \mathbf{i}_r\rightarrow \mathbf{i}_{r+1};\\
			(y_r-y_{r+1})e(\mathbf{i}) &\textit{ if } \mathbf{i}_r\leftarrow \mathbf{i}_{r+1};\\
			(y_{r+1}-y_r)(y_r-y_{r+1})e(\mathbf{i}) &\textit{ if } \mathbf{i}_r\leftrightarrow \mathbf{i}_{r+1}.
		\end{cases}
	\end{gather}
	\begin{gather} \label{26eq}
		(\psi_r\psi_{r+1}\psi_r-\psi_{r+1}\psi_r\psi_{r+1})e(\mathbf{i})=
		\begin{cases}
			e(\mathbf{i}) &\textit{ if } \mathbf{i}_{r+2}=\mathbf{i}_r\rightarrow \mathbf{i}_{r+1};\\
			-e(\mathbf{i}) &\textit{ if } \mathbf{i}_{r+2}=\mathbf{i}_r\leftarrow \mathbf{i}_{r+1};\\
			(y_r-2y_{r+1}+y_{r+2})e(\mathbf{i}) &\textit{ if } \mathbf{i}_{r+2}=\mathbf{i}_r\leftrightarrow \mathbf{i}_{r+1};\\
			0 &\textit{ otherwise } .
		\end{cases}
	\end{gather}
	where $s_r\circ i$ in (\ref{sym}) is the natural action of $\mathfrak{S}_n$ on $I^n$ and the arrows in (\ref{25eq}) and (\ref{26eq}) are those in the quiver $\Gamma$.
\end{defn}

Note that all the relations above are homogeneous with respect to the following degree function on the generators:
\begin{equation}\label{degf}
	\text{deg } e(\mathbf{i})=0,\text{deg } y_r=2\text{ and  deg }\psi_se(\mathbf{i})=-c_{\mathbf{i}_s,\mathbf{i}_{s+1}}
\end{equation}
where $\mathbf{i}\in K^n$,$1\leq r\leq n$ and $1\leq s\leq n-1$. Therefore, the cyclotomic KLR algebra defined above is $\mathbb{Z}$-graded with respect to the degree function in (\ref{degf}).

We next recall the interpretation of the automorphism $\sigma$ as that of the cyclotomic KLR algebra by Rostam in \cite{Rostam_2018}.
Let $\sigma_0:\Gamma_{e,p}\mapsto \Gamma_{e,p}$ be the quiver automorphism given by
\begin{equation*}
	\sigma_0((i,j))=(i-1,j)
\end{equation*}
and $\sigma_1$ be the map on $K^n$ induced by $\sigma_0$, that is the map such that
\begin{equation}\label{8.14}
	\sigma_1(\mathbf{i})_l=\sigma_0(\mathbf{i}_l)
\end{equation}
for all $\mathbf{i}\in K^n$ and $1\leq l\leq n$.
Let $\sigma'$ be the algebra automorphism of the cyclotomic KLR algebra $R_n^{\Lambda}(\Gamma_{e,p})$ given, in terms of the presentation \cite[Def. 2.27]{LL22} by 
\begin{align*}
	\sigma'(e(\mathbf{i}))&=e(\sigma_1(\mathbf{i}));\\
	\sigma'(y_i)=y_i &\text{  for }1\leq i\leq n\\
	\sigma'(\psi_i)=\psi_i &\text{  for } 1\leq i\leq n-1.
\end{align*}

Then we have
\begin{thm} $\label{auto}$
	(\cite{Rostam_2018},Theorem 4.14) 
	Let $f:R_n^{\Lambda}(\Gamma_{e,p})\to H_n^{\Lambda}(q,\zeta)$ be the isomorphism between the cyclotomic KLR algebra and Hecke algebra given in Theorem 3.17 in \cite{Rostam_2018} and $\sigma$ be the isomorphism of $H_n^{\Lambda}(q,\zeta)$ defined in (\ref{sigma}). 
	Let $\sigma'$ be the isomorphism of $R_n^{\Lambda}(\Gamma_{e,p})$ defined above. Then we have $f\circ \sigma'=\sigma\circ f$.
\end{thm}
The following corollary is a direct consequence of this theorem:

\begin{cor}
	The Hecke algebra $H_{p,n}^{\Lambda}(q)$ is isomorphic to $R_n^{\Lambda}(\Gamma_{e,p})^{\sigma}$, the fixed point subalgebra of $R_n^{\Lambda}(\Gamma_{e,p})$ under $\sigma$.
\end{cor}
We shall not distinguish below between the two automorphisms $\sigma$ and $\sigma'$.

\section{The generalised Temperley-Lieb algebra $TL_{r,1,n}$}

In this section, we recall the definition of the generalised Temperley-Lieb algebra $TL_{r,1,n}$ constructed in \cite{LL22}. Our new generalised Temperley-Lieb algebra $TL_{r,p,n}$ will later be defined as the subalgebra of $TL_{r,1,n}$ fixed by automorphism induced by $\sigma$.

Let $H(r,1,n)$ be the cyclotomic Hecke algebra defined in Definition $\ref{Hr}$ and $H_{i,i+1}$ be the subalgebra of $H(r,1,n)$ which is generated by two non-commuting generators, $T_{i}$ and $T_{i+1}$ where $0\leq i\leq n-2$. We call $H_{i,i+1}$ a parabolic subalgebra of $H(r,1,n)$. The parabolic subalgebra $H_{0,1}$ has $2r$ 1-dimensional representations corresponding to the multipartitions $(0,0,\dots,(2),\dots,0)$ and $(0,0,\dots,(1,1),\dots,0)$ where $(2)$ and $(1,1)$ are in the $j^{th}$ component of the $r$-partition with $1\leq j \leq r$. Denote by $E_0^{(j)}$ and $F_0^{(j)}$ the corresponding primitive central idempotents in $H_{0,1}$. These exist because of our choice of parameters, which implies semisimplicity. 

 For $i$ such that $1\leq i\leq n-2$, the parabolic subalgebra $H_{i,i+1}$ is a Hecke algebra of type $A_2$. It has two 1-dimensional representations whose corresponding primitive central idempotents are
\begin{equation*}
  \begin{aligned}
	E_i&=a(T_iT_{i+1}T_i+T_iT_{i+1}+T_{i+1}T_i+T_i+T_{i+1}+1);\\
	F_i&=b(T_iT_{i+1}T_i-qT_iT_{i+1}-qT_{i+1}T_i+q^2T_i+q^2T_{i+1}-q^3),
\end{aligned}  
\end{equation*}

where $a=(q^3+2q^2+2q+1)^{-1}$ and $b=-a$. The generalised Temperley-Lieb algebra $TL_{r,1,n}$ is defined as the quotient of $H(r,1,n)$ by the two-sided ideal generated by half of the central idempotents listed above; specifically, the definition is as follows.

\begin{defn}\label{TL} (\cite{LL22}, Definition 3.3)
	Let $H(r,1,n)$ be the cyclotomic Hecke algebra defined in \ref{Hr} and $E_i$ ($E_i^{(j)}$, if $i=0$) be the primitive central idempotents of $H_{i,i+1}$ listed above. 
	The generalised Temperley-Lieb algebra $TL_{r,1,n}$ is
	\begin{equation*}
		TL_{r,1,n}:=H(r,1,n)/\langle E_0^{(1)},\dots,E_0^{(r)},E_1,\dots,E_{n-2}\rangle.
	\end{equation*}
    
\end{defn}

According to Theorem 1.1 in \cite{Brundan_2009}, the Hecke algebra $H(r,1,n)$ is isomorphic to a cyclotomic KLR algebra $\mathcal{R}_n^{\Lambda}$ with $\Lambda$ a dominant weight determined by the parameters $q,v_1,v_2,\dots,v_r$.
Therefore, we have the following alternative definition of $TL_{r,1,n}$ as a quotient of the cyclotomic KLR algebra $\mathcal{R}_n^{\Lambda}$:

\begin{thm} $\label{dftl}$ \cite[Theorem 3.24]{LL22}
    Let $TL_{r,1,n}$ be the generalised Temperley-Lieb algebra in Definition \ref{TL} and $\mathcal{R}_n^{\Lambda}$ be the cyclotomic KLR algebra isomorphic to the corresponding Hecke algebra $H(r,1,n)$.
	 Then
	\begin{equation}
		TL_{r,1,n}\cong  \mathcal{R}_n^{\Lambda}/\mathcal{J}_{n},
	\end{equation} 
	
	where the ideal $\mathcal{J}_n$ is defined by
%  \begin{equation}\label{ieq}
% 	\begin{aligned}
% 		\mathcal{J}_{n}(\Lambda)=&\sum_{{(i,i+1,\mathbf{i})\in I^{n},\atop(\alpha_i,\Lambda)>0}}\langle  e(i,i+1,\mathbf{i}) \rangle_{\mathcal{R}_{n}^{\Lambda}}\\
% &+\sum_{{(i_1,i_2,i_3,\mathbf{i})\in I^{n},\atop(\alpha_{i_j},\Lambda)>0\text{ for }j=1,2,3,i_j\neq i_k \textit{ for }j\neq k}}\langle  e(i_1,i_2,i_3,\mathbf{i}) \rangle_{\mathcal{R}_{n}^{\Lambda}},
% 	\end{aligned}	
% \end{equation}
% where {$e(i,i+1,\mathbf{i})$ and $e(i_1,i_2,i_3,\mathbf{i})$} are KLR generators of $\mathcal{R}_n^{\Lambda}$
% { [recall that for any $\mathbf{j}\in I^n$, $e(\mathbf{j})$ is the corresponding KLR generator of $\mathcal{R}_n^{\Lambda}$ (cf. Definition \ref{klrdf}).]}
{ \begin{equation}\label{ieq}
 	\begin{aligned}
 		\mathcal{J}_{n}(\Lambda)=&\sum_{\mathbf{i}\in I^{n},\atop\mathbf{i}_2=\mathbf{i}_1+1,(\alpha_{\mathbf{i}_1},\Lambda)>0}\langle  e(\mathbf{i}) \rangle_{\mathcal{R}_{n}^{\Lambda}}\\
&+\sum_{\mathbf{i}\in I^{n},\atop(\alpha_{\mathbf{i}_j},\Lambda)>0\text{ for }j=1,2,3,\mathbf{i}_j\neq \mathbf{i}_k \textit{ for }j\neq k}\langle  e(\mathbf{i}) \rangle_{\mathcal{R}_{n}^{\Lambda}},
 	\end{aligned}	
 \end{equation}
 where $e(\mathbf{i})$ is the corresponding KLR generator of $\mathcal{R}_n^{\Lambda}$ (cf. Definition \ref{klrdf}).}
\end{thm}

\section{The definition of $TL_{r,p,n}$} \label{subsec:813}

In this section, we define the generalised Temperley-Lieb algebra $TL_{r,p,n}$ corresponding to $G(r,p,n)$ using the automorphism $\sigma$. In analogy with the case $p=1$, we will show that $TL_{r,p,n}$ is a quotient of the corresponding Hecke algebra $H(r,p,n)$.

Denote by $TL_{r,1,n}(q,\zeta)$ %the specialisation of 
the generalised Temperley-Lieb algebra $TL_{r,1,n}$ corresponding to $H_n^{\Lambda}(q,\zeta)$. It is a specialisation of the generalised Temperley-Lieb algebra $TL^{\Lambda}_{r,1,n}$.

By Theorem \ref{dftl}, $TL_{r,1,n}(q,\zeta)=\mathcal{R}_n^{\Lambda}(\Gamma_{e,p})/\mathcal{J}_{n}(\Lambda)$ where {$\mathcal{J}_{n}(\Lambda)$} is the two-sided ideal 
generated by all the $e(\mathbf{i}),\mathbf{i}\in K^n$ such that:
\begin{equation}\label{con1}
	i_1=i_2,j_1=j_2-1\text{  if }\mathbf{i}_1=(i_1,j_1)\text{ and }\mathbf{i}_2=(i_2,j_2)
\end{equation}
or
\begin{equation}\label{con2}
	(\alpha_{\mathbf{i}_l},\Lambda)>0\text{ for }l=1,2,3.
\end{equation}

As $\sigma_0((i,j))=(i-1,j)$, $e(\mathbf{i})$ satisfies $(\ref{con1})$ if and only if $\sigma(e(\mathbf{i}))$ does. The restriction on $\Lambda$ implied by $(\Lambda,\alpha_{(i,j)})=(\Lambda,\alpha_{(i',j)})$ guarantees that $e(\mathbf{i})$ satisfies $(\ref{con2})$ if and only if $\sigma(e(\mathbf{i}))$ does. Therefore, we have
\begin{equation}\label{eq:i-sigma-inv}
	\sigma(\mathcal{J}_{n}(\Lambda))=\mathcal{J}_{n}(\Lambda),
\end{equation}
which implies that $\sigma$ induces an automorphism of $TL_{r,1,n}(q,\zeta)$, which we denote by $\sigma_{TL}$. Thus $\sigma=\sigma_{TL}$ is the automorphism of $TL_{r,1,n}(q,\zeta)$ such that
\begin{equation} \label{eq817}
    \begin{aligned}
	\sigma_{TL}(e((i_1,j_1),(i_2,j_2),\dots,(i_n,j_n)))&=e((i_1-1,j_1),(i_2-1,j_2),\dots,(i_n-1,j_n));\\
	&\sigma_{TL}(y_i)=y_i \text{  for }1\leq i\leq n\\
	&\sigma_{TL}(\psi_i)=\psi_i \text{  for } 1\leq i\leq n-1.
\end{aligned}
\end{equation}
We can now state the main theorem of this section.

\begin{thm}$\label{TLdf}$
Let $\zeta\in R$ be a primitive root of unity of order $p$ and let $TL_{r,1,n}(q,\zeta)$ be the specialisation of the generalised Temperley-Lieb algebra 
$TL_{r,1,n}^{\Lambda}$ such that the relation $(\ref{ordering relation})$ is transformed into (\ref{domwt}).
    Let $\sigma_{TL}$ be the automorphism of $TL_{r,1,n}(q,\zeta)$ defined above, let
     $TL_{r,1,n}(q,\zeta)^{\sigma_{TL}}$ be the fixed point subalgebra under $\sigma_{TL}$ and let $\mathcal{J}^{\sigma}=\mathcal{J}_n(\Lambda)\cap \mathcal{R}_n^{\Lambda}(\Gamma_{e,p})^{\sigma}$. Then we have
	\begin{equation}
		TL_{r,1,n}(q,\zeta)^{\sigma_{TL}}\cong \mathcal{R}_n^{\Lambda}(\Gamma_{e,p})^{\sigma}/\mathcal{J}^{\sigma}.
	\end{equation}
	Thus the fixed subalgebra $TL_{r,1,n}(q,\zeta)^{\sigma_{TL}}$ is a quotient of $\mathcal{R}_n^{\Lambda}(\Gamma_{e,p})^{\sigma}$, which is isomorphic to the cyclotomic Hecke algebra of type $G(r,p,n)$.
\end{thm}
\begin{proof}
{By Theorem \ref{dftl}, $TL_{r,1,n}(q,\zeta)=\mathcal{R}_n^{\Lambda}(\Gamma_{e,p})/\mathcal{J}_{n}(\Lambda)$.}
	Let $f:\mathcal{R}_n^{\Lambda}(\Gamma_{e,p})^{\sigma}/\mathcal{J}^{\sigma}$ $\to TL_{r,1,n}(q,\zeta)^{\sigma_{TL}}$ be the map such that $f(a+\mathcal{J}^{\sigma})=a+\mathcal{J}_{n}(\Lambda)$ for all $a\in \mathcal{R}_n^{\Lambda}(\Gamma_{e,p})^{\sigma}$. We only need to prove the surjectivity.
	
%	For $a,b\in \mathcal{R}_n^{\Lambda}(\Gamma_{e,p})^{\sigma}$, if $a+\mathcal{J}^{\sigma}=b+\mathcal{J}^{\sigma}$, $a-b\in \mathcal{J}^{\sigma}\subset \mathcal{J} $ which implies $a+\mathcal{J}=b+\mathcal{J}$. So the map $f$ is well-defined.
	
%	 We next check the injectivity. If $a+\mathcal{J}\neq b+\mathcal{J}$, $a-b\not \in  \mathcal{J}$ thus $a-b\not \in  \mathcal{J}^{\sigma}$. So $a+\mathcal{J}^{\sigma}\neq b+\mathcal{J}^{\sigma}$, which implies that $f$ is an injection.
	 
	 For any $a+\mathcal{J}\in TL_{r,1,n}(q,\zeta)^{\sigma_{TL}}$, we have $a+\mathcal{J}=\sigma(a+\mathcal{J})=\sigma(a)+\mathcal{J}$. Thus, $a+\mathcal{J}=\sigma(a)+\mathcal{J}=\dots=\sigma^{p-1}(a)+\mathcal{J}$. As the characteristic of the field is 0 , $a'=\frac{a+\sigma(a)+\dots+\sigma^{p-1}(a)}{p}$ is an element in $\mathcal{R}_n^{\Lambda}(\Gamma_{e,p})^{\sigma}$ and $a'+\mathcal{J}=a+\mathcal{J}$. Further, $a'+\mathcal{J}^{\sigma}\in \mathcal{R}_n^{\Lambda}(\Gamma_{e,p})^{\sigma}/\mathcal{J}^{\sigma}$ and $f(a'+\mathcal{J}^{\sigma})=a+\mathcal{J}$. So $f$ is a surjection.
	 
%	 It is trivial to check that $f$ is an algebra homomorphism.
	 So $f:\mathcal{R}_n^{\Lambda}(\Gamma_{e,p})^{\sigma}/\mathcal{J}^{\sigma}\mapsto TL_{r,1,n}(q,\zeta)^{\sigma_{TL}}$ is an isomorphism. Therefore, we have $TL_{r,1,n}(q,\zeta)^{\sigma_{TL}}\cong \mathcal{R}_n^{\Lambda}(\Gamma_{e,p})^{\sigma}/\mathcal{J}^{\sigma}$.
\end{proof}
We now define our generalised Temperley-Lieb algebra of type $G(r,p,n)$ as follows:
\begin{defn} \label{dfrpn}
        Let $R$ be a field of characteristic $0$ and let $TL_{r,1,n}(q,\zeta)$ be the generalised Temperley-Lieb algebra defined over $R$ as in \S \ref{subsec:813} above.
        Let $\sigma_{TL}$ be the automorphism of $TL_{r,1,n}(q,\zeta)$ used in Theorem \ref{TLdf}. 
        We define the fixed subalgebra $TL_{r,1,n}(q,\zeta)^{\sigma_{TL}}$ as $TL_{r,p,n}$ and refer to it as the Temperley-Lieb algebra corresponding to the complex reflection group $G(r,p,n)$.
\end{defn}

Since the map $\sigma_{TL}$ is trivial when $p=1$, $TL_{r,1,n}(q,\zeta)$ can be regarded as a special case of the generalised Temperley-Lieb algebra $TL_{r,p,n}$. In particular, the Temperley-Lieb algebras of types $A_{n-1}$ and $B_{n}$ are both special cases of our $TL_{r,p,n}$, where $r=p=1$ and $r=2,\; p=1$, respectively.
As another special case of $TL_{r,p,n}$ with $r=p=2$, the algebra $TL_{2,2,n}$ is a quotient of the Temperley-Lieb algebra of type-$D_n$ in the sense of \cite{Sun2009} and \cite{Li2021}.
This quotient is called a forked Temperley-Lieb algebra. We refer readers to \cite{GROSSMAN2007477} for details.

\section{3-dimensional multipartitions and the cellular structure of $TL_{r,1,n}$} \label{3dmlp}
%According to Theorem $\ref{iso thm}$, the Hecke algebra $H_n^{\Lambda}(q,\zeta)$ is isomorphic to the cyclotomic KLR algebra $R_n^{\Lambda}(\Gamma_{e,p})$ corresponding to the quiver $\Gamma_{e,p}$. Furthermore, $H_n^{\Lambda}(q,\zeta)$ is a cellular algebra according to Theorem $\ref{cb33}$.

As is evident from the above definition, $TL_{r,1,n}(q,\zeta)$ is a specialisation of the Temperley-Lieb algebra $TL_{r,1,n}$. Hence by Theorem 4.19 in \cite{LL22}, there is a natural cellular
structure  on $TL_{r,1,n}(q,\zeta)$.

  In this section, we focus on this specialisation. As we see from the interpretation above, this specialisation transforms the quiver with a single layer in section 2.4 in \cite{LL22} into one with $p$ layers. Accordingly, we need to rearrange the partitions in a multipartition.
More precisely, we introduce 3-dimensional multipartitions and tableaux. They are the primary tool we will later use to prove the cellularity of the generalised Temperley-Lieb algebra $TL_{r,p,n}$.

  Let $\mathcal{K}=I'\times L$ be an index set where $I'=\{0,1,\dots,p-1 \}$ and $L=\{0,1,\dots,d-1 \}$. Note that this index set $\mathcal{K}$ is different from the vertex set $K$ of the quiver $\Gamma_{e,p}$. We first construct a table with $d$ rows and $p$ columns and label the boxes in the table with the elements of $\mathcal{K}$. We call this table the floor of a 3D multipartition. Figure-\ref{floor} shows the floor with $p=5$ and $d=3$:
  
  \begin{figure}[h]
	\begin{center}
		\begin{tikzpicture}[scale=1]
			%% draw straight lines where required
			\foreach \x in {0,1,2,3}
			\draw (0,\x)--(5,\x);
			\foreach \x in {0,1,2,3,4,5}
			\draw (\x,3)--(\x,0);
			
			%% fill in the boxes
			\draw node at (0.5,0.5){(0,0)};\draw node at (1.5,0.5){(1,0)};\draw node at (2.5,0.5){(2,0)};\draw node at (3.5,0.5){(3,0)};\draw node at (4.5,0.5){(4,0)};
						\draw node at (0.5,1.5){(0,1)};\draw node at (1.5,1.5){(1,1)};\draw node at (2.5,1.5){(2,1)};\draw node at (3.5,1.5){(3,1)};\draw node at (4.5,1.5){(4,1)};
									\draw node at (0.5,2.5){(0,2)};\draw node at (1.5,2.5){(1,2)};\draw node at (2.5,2.5){(2,2)};\draw node at (3.5,2.5){(3,2)};\draw node at (4.5,2.5){(4,2)};
		\end{tikzpicture}
		%\centerline{Figure 2}
	\end{center}
	\caption{The floor of $5\times 3$ } \label{floor}
\end{figure}
  
  A {\em 3D multipartition of $n$} consists of a floor labelled with $\mathcal{K}$ and $n(i,l)$ nodes in a column in the box $(i,l)$ on the floor such that 
  \begin{equation*}
      \sum_{i=0}^{p-1}\sum_{l=0}^{d-1}n(i,l)=n.
  \end{equation*}
  Denote by $\mathfrak{P}_n$ the set of 3D multipartitions of $n$ in which at most two boxes on the floor are non-empty. Figures \ref{3dir} and \ref{3dr} give two examples of such 3D multipartitions. We next introduce some notation relating to the 3D multipartitions in $\mathfrak{P}_n$.
  
 For $\lambda\in \mathfrak{P}_n$, if $\lambda$ has only one non-empty component, which is in the box $(i,l)$, call $\lambda$ a single multipartition, and denote it by $\lambda_{(i,l)}$. Otherwise, let $(i_1,l_1)$ and $(i_2,l_2)$ be the two non-empty boxes with $l_1<l_2$ or $l_1=l_2$ and $i_1<i_2$. If there are $x$ nodes in the partition in box $(i_1,l_1)$ and $y$ nodes in $(i_2,l_2)$, denote $\lambda$ by $\lambda_{(i_1,l_1),(i_2,l_2)}^{[x-y]}$. As $\lambda$ is a multipartition of $n$, we have $x+y=n$. Therefore, the multipartition $\lambda$ is determined by the indices $(i_1,l_1),(i_2,l_2)$ and the number ${x-y}$. We are now in a position to define the reducibility of a 3D multipartition.
 \begin{defn}\label{irrr}
     Let $\lambda\in \mathfrak{P}_n$ be a 3D multipartition. If $\lambda$ is of the form $\lambda_{(i_1,l_1),(i_2,l_2)}^{[x-y]}$ with $i_1\neq i_2$, we call it an irreducible multipartition. Otherwise, we call it reducible.
 \end{defn}

 In other words, if the two non-empty boxes are in the same column of the floor, we call the multipartition a reducible one; otherwise, it is irreducible. In the last section of this chapter, we will show that the cell modules of $TL_{r,p,n}$ corresponding to irreducible multipartitions are simple. Figure-\ref{3dir} gives an irreducible multipartition of $5$ and Figure-\ref{3dr} gives a reducible one.
  
   \begin{figure}[h]
	\begin{center}
		\begin{tikzpicture}[scale=1]
			%% draw straight lines where required
			\foreach \x in {0,0.5,1,1.5}
			\draw [dashed] (0+\x,\x)--(5+\x,\x);
			\foreach \x in {0,1,2,3,4,5}
			\draw [dashed] (\x+1.5,1.5)--(\x,0);
			
			\draw[blue, very thick] (2,0)--(2,2);\draw[blue, very thick] (3,0)--(3,2);\draw[blue, very thick] (2.5,0.5)--(2.5,2.5);\draw[blue, very thick] (3.5,0.5)--(3.5,2.5);
			\draw[blue, very thick] (2,0)--(3,0);\draw[blue, very thick] (2.5,0.5)--(3.5,0.5);\draw[blue, very thick] (2,0)--(2.5,0.5);\draw[blue, very thick] (3,0)--(3.5,0.5);
			\draw[blue, very thick] (2,1)--(3,1);\draw[blue, very thick] (2.5,1.5)--(3.5,1.5);\draw[blue, very thick] (2,1)--(2.5,1.5);\draw[blue, very thick] (3,1)--(3.5,1.5);
			\draw[blue, very thick] (2,2)--(3,2);\draw[blue, very thick] (2.5,2.5)--(3.5,2.5);\draw[blue, very thick] (2,2)--(2.5,2.5);\draw[blue, very thick] (3,2)--(3.5,2.5);
			
			\draw[orange, very thick] (5,1)--(5,4);\draw[orange, very thick] (6,1)--(6,4);\draw[orange, very thick] (5.5,1.5)--(5.5,4.5);\draw[orange, very thick] (6.5,1.5)--(6.5,4.5);
			\draw[orange, very thick] (5,1)--(6,1);\draw[orange, very thick] (5.5,1.5)--(6.5,1.5);\draw[orange, very thick] (6,1)--(6.5,1.5);\draw[orange, very thick] (5,1)--(5.5,1.5);
			\draw[orange, very thick] (5,2)--(6,2);\draw[orange, very thick] (5.5,2.5)--(6.5,2.5);\draw[orange, very thick] (6,2)--(6.5,2.5);\draw[orange, very thick] (5,2)--(5.5,2.5);
			\draw[orange, very thick] (5,3)--(6,3);\draw[orange, very thick] (5.5,3.5)--(6.5,3.5);\draw[orange, very thick] (6,3)--(6.5,3.5);\draw[orange, very thick] (5,3)--(5.5,3.5);
			\draw[orange, very thick] (5,4)--(6,4);\draw[orange, very thick] (5.5,4.5)--(6.5,4.5);\draw[orange, very thick] (6,4)--(6.5,4.5);\draw[orange, very thick] (5,4)--(5.5,4.5);

			%% fill in the boxes
            	\draw node at (0.5,-0.5){0};\draw node at (1.5,-0.5){1};\draw node at (2.5,-0.5){2};\draw node at (3.5,-0.5){3};\draw node at (4.5,-0.5){4};
            	\draw node at (-0.5,0.25){0};\draw node at (0,0.75){1};\draw node at (0.5,1.25){2};
		\end{tikzpicture}
		%\centerline{Figure 2}
	\end{center}
	\caption{ The irreducible multipartition $\lambda_{(2,0),(4,2)}^{[-1]}$ } \label{3dir}
\end{figure}

 \begin{figure}[h]
	\begin{center}
		\begin{tikzpicture}[scale=1]
			%% draw straight lines where required
			\foreach \x in {0,0.5,1,1.5}
			\draw [dashed] (0+\x,\x)--(5+\x,\x);
			\foreach \x in {0,1,2,3,4,5}
			\draw [dashed] (\x+1.5,1.5)--(\x,0);
			
			\draw [blue, very thick](2,0)--(2,1);\draw[blue, very thick] (3,0)--(3,1);\draw[blue, very thick] (2.5,0.5)--(2.5,1.5);\draw[blue, very thick] (3.5,0.5)--(3.5,1.5);
			\draw[blue, very thick] (2,0)--(3,0);\draw[blue, very thick] (2.5,0.5)--(3.5,0.5);\draw[blue, very thick] (2,0)--(2.5,0.5);\draw[blue, very thick] (3,0)--(3.5,0.5);
			\draw[blue, very thick] (2,1)--(3,1);\draw[blue, very thick] (2.5,1.5)--(3.5,1.5);\draw[blue, very thick] (2,1)--(2.5,1.5);\draw[blue, very thick] (3,1)--(3.5,1.5);

			\draw[orange, very thick] (3,1)--(3,5);\draw[orange, very thick] (4,1)--(4,5);\draw[orange, very thick] (3.5,1.5)--(3.5,5.5);\draw[orange, very thick] (4.5,1.5)--(4.5,5.5);
			\draw[orange, very thick] (3,1)--(4,1);\draw[orange, very thick] (3.5,1.5)--(4.5,1.5);\draw[orange, very thick] (4,1)--(4.5,1.5);\draw[orange, very thick] (3,1)--(3.5,1.5);
			\draw[orange, very thick] (3,2)--(4,2);\draw[orange, very thick] (3.5,2.5)--(4.5,2.5);\draw[orange, very thick] (4,2)--(4.5,2.5);\draw[orange, very thick] (3,2)--(3.5,2.5);
			\draw[orange, very thick] (3,3)--(4,3);\draw[orange, very thick] (3.5,3.5)--(4.5,3.5);\draw[orange, very thick] (4,3)--(4.5,3.5);\draw[orange, very thick] (3,3)--(3.5,3.5);
			\draw[orange, very thick] (3,4)--(4,4);\draw[orange, very thick] (3.5,4.5)--(4.5,4.5);\draw[orange, very thick] (4,4)--(4.5,4.5);\draw[orange, very thick] (3,4)--(3.5,4.5);
			\draw[orange, very thick] (3,5)--(4,5);\draw[orange, very thick] (3.5,5.5)--(4.5,5.5);\draw[orange, very thick] (4,5)--(4.5,5.5);\draw[orange, very thick] (3,5)--(3.5,5.5);

			%% fill in the boxes
                        	\draw node at (0.5,-0.5){0};\draw node at (1.5,-0.5){1};\draw node at (2.5,-0.5){2};\draw node at (3.5,-0.5){3};\draw node at (4.5,-0.5){4};
            	\draw node at (-0.5,0.25){0};\draw node at (0,0.75){1};\draw node at (0.5,1.25){2};
		\end{tikzpicture}
		%\centerline{Figure 2}
	\end{center}
	\caption{ The reducible multipartition $\lambda_{(2,0),(2,2)}^{[-3]}$ } \label{3dr}
\end{figure}

  We next define the residue of each node in a multipartition. This residue provides a connection between the multipartitions and the quivers we introduced above.
  Let {$j_1, \ldots, j_{d-1} \in J=\mathbb{Z}/e\mathbb{Z}$} be such that the dominant weight $\Lambda$ associated with the parameters in (\ref{domwt}) is of the following form
  \begin{equation}
  	\Lambda=\sum_{i=0}^{p-1}\sum_{l=0}^{d-1}\Lambda_{(i,j_l)}.
  \end{equation}
  For $\lambda\in\mathfrak{P}_n$, let $\gamma=(a,1,(i,l))$ be the node in the $a^{th}$ row of the $(i,l)^{th}$ component of $\lambda$. We denote
  \begin{equation*}
  	Res^{\Lambda}(\gamma)=(i,1-a+j_l)\in K
  \end{equation*}
where $K=\mathbb{Z}/p\mathbb{Z}\times \mathbb{Z}/e\mathbb{Z}$ is the vertex set of the quiver $\Gamma_{e,p}$.

A tableau corresponding to a 3D multipartition of $n$ is defined as a {bijective} filling of the multipartition with numbers $1,2,\dots, n$; we call it standard if these numbers increase from the floor table along each column.
For a 3D multipartition $\lambda$, denote by $Tab (\lambda)$ the set of tableaux of shape $\lambda$ and by $Std(\lambda)$ the standard ones. For $t\in Tab (\lambda)$ and $1\leq m\leq n$, set $Res_t^{\Lambda}(m)=Res^{\Lambda}(\gamma)$, where $\gamma$ is the unique node such that $t(m)=\gamma$. Define the residue sequence of $t$ as follows:
\begin{equation} \label{5.3}
	res^{\Lambda}(t):=(Res_t^{\Lambda}(1),Res_t^{\Lambda}(2),\dots,Res_t^{\Lambda}(n))\in K^n.
\end{equation}

For two nodes $\gamma=(a,1,(i,l))$ and $\gamma'=(a',1,(i',{l'}))$ in a 3D multipartition $\lambda$ we say $\gamma <\gamma'$ if one of the following conditions holds:

  (i) $a<a'$;
 
  (ii) $a=a'$ and $l<l'$;
  
  (iii)$a=a'$, $l=l'$ and $i<i'$. 
  
  We write $\gamma \leq \gamma'$ if $\gamma<\gamma'$ or $\gamma=\gamma'$. The order $\leq$ is a total order on the set of nodes. Denote by $t^{\lambda}$ the unique tableau of shape $\lambda$ such that $t^{\lambda}(i)<t^{\lambda}(j)$ if $1\leq i<j\leq n$. Let $e_{\lambda}=e(res^{\Lambda}(t^{\lambda}))$ be the {corresponding} idempotent in $R_n^{\Lambda}(\Gamma_{e,p})$. 

  Now the total order $\leq$ on nodes leads to a partial order $\unlhd $ on $\mathfrak{P}_n$, the set of 3D-multipartitions in the following way: for $\lambda$ and $ \mu \in \mathfrak{P}_n $, we say $\lambda\unlhd \mu$ if for each $\gamma_0\in \mathbb{N}\times \{1\}\times \mathcal{K}$ we have
\begin{equation}\label{podf}
	|\{\gamma\in[\lambda]:\gamma \leq \gamma_0\}|\geq 	|\{\gamma\in[\mu ]:\gamma \leq \gamma_0\}|,
\end{equation}
     where $\leq$ is the total order of nodes. 
  
  Let $s,t\in Std(\lambda)$ and fix reduced expressions $d(s)=s_{i_1}s_{i_2}\dots s_{i_k}$ and $d(t)=s_{j_1}s_{j_2}\dots s_{j_m}$ such that $s=t^{\lambda}\circ d(s)$ and $t=t^{\lambda}\circ d(t)$ where $\circ$ is the natural $\mathfrak{S}_n$-action on the tableaux. Let $*$ be the unique $R$-linear anti-automorphism of the KLR algebra $\mathcal{R}_n^{\Lambda}$ introduced by Brundan and Kleshchev in section 4.5 of \cite{BrundanKleshchev_2009} which fixes all the generators of KLR type.

  For any 3D-multipartition $\lambda\in \mathfrak{P}_n$ and $s,t\in Std(\lambda)$, let
  \begin{equation}\label{eq104}
  		C_{s,t}^{\lambda}=\psi_{d(s)}^* e_{\lambda}\psi_{d(t)}.
  \end{equation}

{The rules (ii) and (iii) above give a total order of the positions on the floor and the derived partial order $\unlhd$ is equivalent to the partial order $\unlhd$ in \cite{LL22} and the partial order $\unlhd_0$ in \cite{LOBOSMATURANA2020106277}. By giving up the 3D construction, we can obtain the following cellular basis:}

\begin{thm}(\cite{LL22}, Theorem 4.19) \label{cb419}
    The set $\{ C_{s,t}^{\lambda}| \lambda\in \mathfrak{P}_n, s,t\in Std(\lambda)\}$ forms a (graded) cellular basis of the generalised Temperley-Lieb algebra $TL_{r,1,n}(q,\zeta)$ with respect to the partial order $\unlhd $.
\end{thm}

  \section{Skew cellular algebras}
In this section, we review some general results on graded skew cellular algebras, which were introduced by Hu, Mathas and Rostam in \cite{hu2021skew} as a
 generalisation of the cellular algebras by Graham and Lehrer in \cite{Lehrer1996}. We will demonstrate the cellularity of our generalised Temperley-Lieb algebra $TL_{r,p,n}$ using the theory of skew cellularity.

In a similar way to cellular algebras, skew cellular algebras are defined in terms of a skew cell datum, which consists of the terms of a cell datum as well as a poset involution.

\begin{defn}
	Let $\mathfrak{P}$ be a finite poset with $\trianglelefteq$ as the partial order. A poset automorphism of $(\mathfrak{P},\trianglelefteq)$ is a permutation $\sigma$ of $\mathfrak{P}$ such that
	\begin{equation*}
		\lambda\trianglelefteq\mu \text{ if and only if }\sigma(\lambda)\trianglelefteq\sigma(\mu )
	\end{equation*}
    for all $\lambda,\mu\in\mathfrak{P}$. If $\sigma=\iota$ is an involution, we say that $\iota$ is a poset involution of $\mathfrak{P}$.
\end{defn}

Graded skew cellular algebras are defined as follows.

\begin{defn}

	(\cite{hu2021skew}, Definition 2.2) Let $R$ be an integral domain and $A$ be a $\mathbb{Z}$-graded algebra over $R$ which is a free module of finite rank over $R$. 
	
	$A$ is a graded skew cellular algebra if it has a graded skew cellular datum $(\mathfrak{P},\iota,T,C,deg)$ where $(\mathfrak{P},\trianglelefteq)$ is a poset, $\iota$ is a poset involution and for each $\lambda\in\mathfrak{P}$, there is a finite set $T(\lambda)$ together with a bijection $\iota_{\lambda}: T(\lambda)\mapsto T(\iota(\lambda))$ such that
	\begin{equation*}
		\iota_{\iota_{\lambda}}\circ \iota_{\lambda}=id_{T(\lambda)}.
	\end{equation*}
    The injection 
    \begin{align*}
    	C: \cup_{\lambda \in \mathfrak{P}}T(\lambda)\times T(\lambda) & \mapsto A \\
    	(S,T) &\mapsto C_{S,T}^{\lambda}
    \end{align*}
    and the degree map
    \begin{equation*}
    	deg:  \cup_{\lambda \in \mathfrak{P}}T(\lambda)\mapsto \mathbb{Z}
    \end{equation*}
    satisfy the following conditions:
    
    ($C_1$) $C_{S,T}^{\lambda}$ is homogeneous of degree $deg(S)+deg(T)$;
    
    ($C_2$)The image of $C$ forms an $R$-basis of $A$;
    
    ($C_3$) Let $A_{\triangleleft\lambda}$ be the $R$-span of all the elements of form $C_{X,Y}^{\mu}$ with $\mu \triangleleft\lambda$ in the poset.
    Then for all $a\in A$,
    \begin{equation}\label{eq 106}
    	aC_{S,T}^{\lambda}=\sum_{S'\in T(\lambda)} r_a(S',S)C_{S',T}^{\lambda}  \text{    } mod \text{   }A_{\triangleleft\lambda}
    \end{equation}
    with the coefficients $r_a(S',S)$ independent of $T$.
    
    ($C_4$) There is a unique anti-isomorphism $*:A\mapsto A$ such that
    \begin{equation}
    (C_{S,T}^{\lambda})^*=C_{\iota_{\lambda}(T),\iota_{\lambda}(S)}^{\iota(\lambda)}.
    \end{equation} 

    The basis $\{C_{S,T}^{\lambda}|\lambda\in\mathfrak{P},S,T\in T(\lambda) \}$ is called a $\mathbb{Z}$-graded skew cellular basis of $A$.
\end{defn}
    
The following is a baby example of the concept of a skew cellular structure. 
\begin{exmp}
(\cite[Example 2.5]{hu2021skew})
    Let $R$ be a ring and $x,y$ be two indeterminates. For any integer $m\geq 1$, let $A=R[x]/(x^m)\oplus R[y]/(y^m)$. Let $\mathfrak{P}=\mathbb{Z}_2\times \{0,1,\dots,m-1\}$ and $\trianglelefteq$ be the partial order on $\mathfrak{P}$ such that $(i_1,k_1)\trianglelefteq (i_2,k_2)$ only if $i_1=i_2$ and $k_1\leq k_2$. Let $\iota$ be the poset involution on $\mathfrak{P}$ such that $\iota(i,k)=(i+1,k)$. Define $\mathcal{T}(i,k)=\{k\}$ and $deg(k)=k$, that is, there exists a unique `tableau' of each element in $\mathfrak{P}$. Define
    \begin{equation*}
        C_{kk}^{(i,k)}=
        \begin{cases}
        x^k &\textit{ if } i=0;\\
        y^k &\textit{ if } i=1.
        \end{cases}
    \end{equation*}
    Then $(\mathfrak{P},\iota,\mathcal{T},C,deg)$ is a graded skew cellular datum for $A$.
\end{exmp}

By applying the anti-isomorphism $*$ to ($\ref{eq 106}$) and relabelling, we obtain
    \begin{equation}\label{eq 106*}
	C_{S,T}^{\lambda}a=\sum_{T'\in T(\lambda)} r_{a^*}(\iota_{\lambda}(T'),\iota_{\lambda}(T))C_{S,T'}^{\lambda}  \text{    } mod \text{   }A_{\triangleleft\lambda}
\end{equation}
  where the coefficients $r_{a^*}(\iota_{\lambda}(T'),\iota_{\lambda}(T))$ are the same ones as in ($\ref{eq 106}$) and do not depend on $S$.
  
 Note that a graded skew cellular algebra is graded cellular when the poset involution $\iota$ coincides with $id_{\mathfrak{P}}$ and $\iota_{\lambda}=id_{T(\lambda)}$.
 It will later transpire that this observation may be used to show that the skew cellularity of our $TL_{r,p,n}$ is actually cellularity.

  Similarly to cellular structures, skew cellular ones also constitute a useful tool to study the representations of algebras which are not semisimple. Since our generalised Temperley-Lieb algebra $TL_{r,p,n}$ will later turn out to be a cellular algebra, we omit details of skew cell modules and decomposition matrices. We refer the reader to \cite{hu2021skew} for further information.

   Next, recall the definition of shift automorphisms which provide a general method for the construction of skew cellular algebras from cellular ones. 
    \begin{defn}
        
 $\label{shift}$
    (\cite[Definition 2.22]{hu2021skew})
    	Let $A$ be a $\mathbb{Z}$-graded cellular algebra with graded cell datum $(\mathfrak{P},T,C,deg)$. A shift automorphism of $A$ is a triple of automorphisms $\sigma=(\sigma_A,\sigma_{\mathfrak{P}},\sigma_\mathcal{T})$ where $\sigma_A$ is an algebra automorphism of $A$, $\sigma_{\mathfrak{P}}$ is a poset automorphism of $\mathfrak{P}$ and $\sigma_\mathcal{T}$ is a bijection on $\mathcal{T}=\bigsqcup_{\lambda\in \mathfrak{P}}T(\lambda)$ such that:
    	
    	($a$) For $S\in T(\lambda)$, $\sigma_\mathcal{T}(S)\in T(\sigma_{\mathfrak{P}}(\lambda))$ and $deg(\sigma_\mathcal{T}(S))=deg(S)$;
    	
    	($b$) For $S,T\in T(\lambda)$, $\sigma_A (C_{S,T}^{\lambda})=C_{\sigma_\mathcal{T}(T),\sigma_\mathcal{T}(S)}^{\sigma_{\mathfrak{P}}(\lambda)}$;
    	
    	($c$) For $S,T\in T(\lambda)$ {and $k\in \mathbb{N}$}, $\sigma_\mathcal{T}^k(T)=T$ if and only if $\sigma_\mathcal{T}^k(S)=S$.
    \end{defn}

Before stating the main theorem of this section, we introduce some notation.
    
     Denote by $A^{\sigma}$ the subalgebra of $A$ consisting of $\sigma_A$-fixed points. Let $p$ be the order of $\sigma_{A}$ and $p'$ be the order of $\sigma_{\mathfrak{P}}$. Let $\mathfrak{P}_{\sigma}$ be a set of representatives for the $\langle \sigma_{\mathfrak{P}}\rangle$-orbits in $\mathfrak{P}$. Define a partial order $\trianglelefteq_{\sigma}$ on $\mathfrak{P}_n$ by
     \begin{equation*}
     	\lambda\triangleleft_{\sigma}\mu \text{ if and only if } \sigma_{\mathfrak{P}}^k(\lambda)\triangleleft\mu \text{ for some } k\in \mathbb{Z}.
     \end{equation*}
     Denote by $o_{\lambda}$ the size of the $\langle \sigma_{\mathfrak{P}}\rangle$-orbit through $\lambda$. Let $\sigma_{\lambda}=\sigma_{\mathcal{T}}^{o_{\lambda}}$. Similarly, let $T_{\sigma}(\lambda)$ be a set of representatives for the $\langle \sigma_{\lambda}\rangle$-orbit on $T(\lambda)$. The condition ($c$) above implies that for a $\lambda\in \mathfrak{P}$, all the orbits in $T(\lambda)$ are of the same size. Denote this common value by $o_T(\lambda)$. Then $o_{\lambda}=o_{\mu}$ and $o_T(\lambda)=o_T(\mu)$ for $\lambda$ and $\mu$ in the same $\langle \sigma_{\mathfrak{P}}\rangle$-orbit on $\mathfrak{P}$.
     
     Let $\mathfrak{P}_{\sigma,p}:=\{(\lambda,k)|\lambda\in\mathfrak{P}_{\sigma}, k\in\mathbb{Z}/o_T(\lambda)\mathbb{Z} \}$ be the poset with partial order $\triangleleft_{\sigma}$ given by
     \begin{equation*}
     	(\lambda,k)\trianglelefteq_{\sigma}(\mu,l) \text{ if and only if } (\lambda,k)=(\mu,l)\text{ or }\lambda\triangleleft_{\sigma}\mu,
     \end{equation*}
     for all $(\lambda,k),(\mu,l)\in \mathfrak{P}_{\sigma,p}$.
     
     Define $T_{\sigma}(\lambda,k)=T_{\sigma}(\lambda)$ for $(\lambda,k)\in \mathfrak{P}_{\sigma,p}$. Assuming there exists a  primitive $p^{th}$ root of unity {$\zeta \in R$}, set
     \begin{equation*}
     	C_{\sigma}(S,T)=\sum_{j=0}^{o_T(\lambda)-1}\sum_{l=0}^{p-1}\zeta_{\lambda}^{kj}\sigma_{A}^l(C_{S,\sigma_{\lambda}^j(T)}^{\lambda})
     \end{equation*}
      where $\zeta_{\lambda}=\zeta^{p/o_T(\lambda)}$ and $S,T\in T_{\sigma}(\lambda,k)$. Let $deg_{\sigma}(S)=deg(S)$. Finally, let $\iota_{\sigma}$ be the poset involution such that $\iota_{\sigma}(\lambda,k)=(\lambda,-k)$ and $(\iota_{\sigma})_{\lambda,k}:T_{\sigma}(\lambda,k)\mapsto T_{\sigma}(\lambda,-k)$ be the map given by the identity map on $T_{\sigma}(\lambda)$.
     
     Then we have a shift automorphism that leads to a skew cellular structure. More precisely, we have
     \begin{thm} $\label{skc}$
     	(\cite[Theorem 2.28]{hu2021skew}) Suppose $A$ is a $\mathbb{Z}$-graded cellular algebra with graded cell datum $(\mathfrak{P},T,C,deg)$ over the integral domain $R$. Let $\sigma=(\sigma_A,\sigma_{\mathfrak{P}},\sigma_\mathcal{T})$ be a shift automorphism of $A$. Denote by $p$ the order of $\sigma_A$. If $R$ contains a  primitive  $p^{th}$ root of unity and $p\in R^*$, then $A^{\sigma}$ is a graded skew cellular algebra with skew cellular datum $(\mathfrak{P}_{\sigma,p},\iota_{\sigma},T_{\sigma},C_{\sigma},deg_{\sigma})$.
     \end{thm}

 \section{The skew cellularity of $TL_{r,p,n}$} $\label{sec7}$
 
 In \cite{hu2021skew}, Hu, Mathas and Rostam introduced skew cellular algebras as a generalisation of the original cellular algebras. They define the notion of a ``shift automorphism'' of a cellular algebra and show that the fixed subalgebra has a skew cellular structure (see Theorem \ref{skc} above).
    In this section, we prove that $TL_{r,p,n}$ is skew cellular by constructing a shift automorphism on  $TL_{r,1,n}(q,\zeta)$ and identifying $TL_{r,p,n}$ as the subalgebra consisting of the fixed points of the shift automorphism. 
    
    We first define the shift automorphism on the specialised Temperley-Lieb algebra $TL_{r,1,n}(q,\zeta)$.
    As shown in Theorem $\ref{TLdf}$, $TL_{r,p,n}$ is the fixed-point subalgebra of $TL_{r,1,n}(q,\zeta)$ under $\sigma_{TL}$, which is induced by the automorphism $\sigma'$ of the cyclotomic KLR algebra $R_n^{\Lambda}(\Gamma_{e,p})$ defined above Theorem $\ref{auto}$. So we choose the algebra automorphism $\sigma_{TL}$ as the first term in the shift automorphism $\sigma$. 
    
    Condition (b) in Definition \ref{shift} indicates that the poset automorphism $\sigma_{\mathfrak{P}}$ and the bijection $\sigma_{\mathcal{T}}$ are determined by the cellular structure of $A$ and the algebra automorphism $\sigma_{A}$. However, the poset map $\sigma_{\mathfrak{P}}$ induced by the algebra automorphism $\sigma_{TL}$ with respect to the cellular structure in Theorem \ref{cb419} is not an automorphism of the poset  $(\mathfrak{P}_n,\unlhd)$ with the partial order $\unlhd$ given by (\ref{podf}). As a counterexample, let $\lambda_{(0,0)}$ and $\lambda_{(1,0)} \in \mathfrak{P}_n$ be the 3D multipartitions of $n$ with the $(0,0)^{th}$ and $(1,0)^{th}$ components are the only non-empty ones respectively. We have $\lambda_{(0,0)} \unlhd \lambda_{(1,0)}$ but 
    \begin{equation*}
        \sigma_{\mathfrak{P}}(\lambda_{(1,0)})=\lambda_{(0,0)} \unlhd \lambda_{(p-1,0)}=\sigma_{\mathfrak{P}}(\lambda_{(0,0)}).
    \end{equation*}
   
   We therefore introduce a new cellular structure on the algebra $TL_{r,1,n}(q,\zeta)$. It should be pointed out that the following construction cannot be carried out
    for the original $TL_{r,1,n}$ without specialisation to the case we treat.
    We first introduce a new partial order on the indices on the floor table.
    
    \begin{defn} $\label{2po}$
    	Let $\mathcal{K}=\{0,1,\dots,p-1 \}\times\{0,1,\dots,d-1 \}$ be the index set introduced in \S $\ref{3dmlp}$. Define a total order $\leq $ on $\mathcal{K}$, 
	by stipulating that for $(i_1,l_1),(i_2,l_2)\in \mathcal{K}$:
    	\begin{equation}
    		(i_1,l_1)\leq (i_2,l_2) \text{ if and only if } l_1< l_2 \text{  or } l_1=l_2 \text{ and }i_1\leq i_2.
    	\end{equation}
      Further, define the partial order $\leq'$ on  $\mathcal{K}$, as the unique one satisfying that for $(i_1,l_1),(i_2,l_2)\in \mathcal{K}$:
        \begin{equation}
        	(i_1,l_1)\leq' (i_2,l_2) \text{ if and only if } l_1< l_2  \text{  or } l_1=l_2 \text{ and }i_1= i_2.
        \end{equation}
    \end{defn}
    
     It is obvious that $(i_1,l_1)\leq' (i_2,l_2)$ implies $(i_1,l_1)\leq (i_2,l_2)$. Similarly to the notation in \S $\ref{3dmlp}$, denote a multipartition $\lambda\in \mathfrak{P}_n$ by:
     
     (1) $\lambda_{(i,l)}$, if the $(i,l)^{th}$ component is the only non-empty one in $\lambda$. Equivalently, $\lambda$ consists of a column of $n$ nodes in the $(i,l)^{th}$ position.
     
     (2) $\lambda_{(i_1,l_1),(i_2,l_2)}^{[a]}$, if the $(i_1,l_1)^{th}$ and $(i_2,l_2)^{th}$ components in $\lambda$ are non-empty with $(i_1,l_1)<(i_2,l_2)$ in $(\mathcal{K},\leq )$ and $a=a_{i_1,l_1}-a_{i_2,l_2}$ where $a_{i,l}$ is the number of nodes in the $(i,l)^{th}$ component. It is obvious that $2-n\leq a\leq n-2$ and $a\equiv n$ ($mod$ 2).
     
     We next define two partial orders $\trianglelefteq $ and $\trianglelefteq'$ on $\mathfrak{P}_n$, which respectively correspond to $\leq $ and $\leq' $ on $\mathcal{K}$, as follows:
     
     \begin{defn} $\label{2PO}$
      Define the  two partial orders $\trianglelefteq$ and $\trianglelefteq'$ on the set of 3D multipartitions of $n$ to be the unique partial orders satisfying:
     	
     	(1) If $\lambda$ and $\mu$ are of the form $\lambda_{(i_1,l_1)}$ and $\mu_{(i_3,l_3)}$, then
     		\begin{equation*}
		\begin{aligned}
     			&\lambda_{(i_1,l_1)}\trianglelefteq \mu_{(i_3,l_3)} (resp.\; \lambda_{(i_1,l_1)}\trianglelefteq' \mu_{(i_3,l_3)}) \\
			&\text{  if and only if } (i_1,l_1)\leq  (i_3,l_3) (\;resp.\; (i_1,l_1)\leq'  (i_3,l_3))\text{ in } \mathcal{K}.\\
			\end{aligned}
     		\end{equation*}
     (2) If $\lambda$ and $\mu$ are respectively of the form $\lambda_{(i_1,l_1),(i_2,l_2)}^{[a]}$ and $\mu_{(i_3,l_3)}$, then
     
     	\begin{equation*}
	\begin{aligned}
     		&\lambda_{(i_1,l_1)(i_2,l_2)}^{[a]}\trianglelefteq\mu_{(i_3,l_3)} (resp.\;\lambda_{(i_1,l_1)(i_2,l_2)}^{[a]} \trianglelefteq'\mu_{(i_3,l_3)}) \\
		&\text{  if and only if } (i_1,l_1)\leq  (i_3,l_3) (\;resp.\;  (i_1,l_1) \leq' (i_3,l_3) )\text{ in } \mathcal{K}.\\
		\end{aligned}
     	\end{equation*}
	{ Note that if the multipartitions $\lambda_{(i_1,l_1),(i_2,l_2)}^{[a]}$ and $\mu_{(i_3,l_3)}$ in this case are reversed, there is no relation $\mu_{(i_3,l_3)}\trianglelefteq\lambda_{(i_1,l_1),(i_2,l_2)}^{[a]}$ or $\mu_{(i_3,l_3)}\trianglelefteq' \lambda_{(i_1,l_1),(i_2,l_2)}^{[a]}$.}
	
        (3) If $\lambda$ and $\mu$ are respectively of the form $\lambda_{(i_1,l_1),(i_2,l_2)}^{[a]}$ and $\mu_{(i_3,l_3),(i_4,l_4)}^{[b]}$, then
        
        \begin{equation*}
	\begin{aligned}
     		\lambda_{(i_1,l_1)(i_2,l_2)}^{[a]}&\trianglelefteq\mu_{(i_3,l_3), (i_4,l_4)}^{[b]} (resp.\;\lambda_{(i_1,l_1)(i_2,l_2)}^{[a]} \trianglelefteq'\mu_{(i_3,l_3),(i_4,l_4)}^{[b]}) \text{  if and only if }\\
		& (i_1,l_1)\leq  (i_3,l_3), (i_2,l_2)\leq(i_4,l_4) (\;resp.\;  (i_1,l_1) \leq' (i_3,l_3), (i_2,l_2)\leq'(i_4,l_4)) \text{ in } \mathcal{K},\\
		\end{aligned}
     	\end{equation*}
	 and in addition one of the following holds:
        	
        	($i$) $|a|<|b|$;
        	
        	($ii$) $|a|=|b|$ and $a\geq b$;
        	
        	($iii$) $|a|=|b|$, $a<b$ and $(i_2,l_2)\leq  (i_3,l_3)$  $(\;resp.\;(i_2,l_2)\leq' (i_3,l_3))$.

     \end{defn}

%\ml{The following part is added to explain why we can use the results in \cite{LL22}}
     
   Comparing this definition and Lemmas 2.5,2.6 and 2.7 in \cite{LL22}, we have
     
  \begin{cor} \label{poeq}
        The partial order $\trianglelefteq$  defined above is equivalent to the partial order defined in \cite{LL22} which leads to the cellular structure of $TL_{r,1,n}(q,\zeta)$ in Theorem \ref{cb419}.  
    \end{cor}

   % \ml{end}

    %By Lemmas \ref{25},\ref{26} and \ref{27}, the partial order $\trianglelefteq$ is equivalent to the partial order we defined in \ref{cha:tableaux} and leads to the cellular structure of $TL_{r,1,n}(q,\zeta)$ in Theorem \ref{cb33}. 
    Theorem \ref{cb419} implies $\{ C_{s,t}^{\lambda}\}$ is a cellular basis with respect to the partial order $\trianglelefteq$.
    We next use the second partial order $\trianglelefteq'$ to define a new cellular structure on $TL_{r,1,n}(q,\zeta)$.
     As a direct consequence of the fact that $\leq'$ is covered by $\leq$, $\lambda\trianglelefteq' \mu$ implies $\lambda\trianglelefteq \mu$ for any $\lambda,\mu \in \mathfrak{P}_n$. The converse is not always true. Nevertheless, we have the following lemmas:

%  \mll{The following lemma indicates how to obtain $\trianglelefteq'$ with $\trianglelefteq$. This will be used to explain the third last main correction from the referees}

% \begin{lem}
%     Let $\lambda,\mu,\nu \in \mathfrak{P}_n$ be three 3D multipartitions with $\mu\neq \lambda$ such that $\nu \trianglelefteq \mu$ and $\mu \trianglelefteq' \lambda$. We have $\nu \trianglelefteq' \lambda$.
% \end{lem}
% \begin{proof}
%     We first prove that $(i_1,l_1)<' (i_3,l_3)$, if $(i_1,l_1)\leq (i_2,l_2)$ and $(i_2,l_2)<' (i_3,l_3)$ for $(i_1,l_1),(i_2,l_2),(i_3,l_3)\in \mathcal{K}$.
% \end{proof}
% \mll{end}
     
     \begin{lem} $\label{weaklm}$
     	Let $\lambda,\mu\in\mathfrak{P}_n$ be two 3D multipartitions satisfying $\lambda\mla{\mla{\triangleleft}} \mu \in \mathfrak{P}_n$. 
	If there exist a standard tableau $t$ of shape $\lambda$ and a standard tableau $s$ of shape $\mu$ such that
     	\begin{equation*}
     		e(t)=e(s),
     	\end{equation*}
     	 then we have $\lambda\mla{\mla{\triangleleft'}} \mu$ in $\mathfrak{P}_n$ where $\mla{\mla{\triangleleft'}}$ is the finer partial order defined above.
     \end{lem}
     \begin{proof}
     	As $\mu \in \mathfrak{P}_n$, it has one of the following three shapes:
     	
     	(a) $\mu$ has only one non-empty component, that is $\mu=\mu_{(i,l)}$. In this case, we have $s$ is the unique standard tableau $t^{\mu}$ and
     	\begin{equation}
     		res^{\Lambda}(t^{\mu})=((i,j_l),(i,j_l-1),\dots,(i,j_l-n)).
     	\end{equation}
         So $res^{\Lambda}(t)=res^{\Lambda}(t^{\mu})=((i,j_l),(i,j_l-1),\dots,(i,j_l-n))$, which implies that the shape of $t$ is either of the form $\lambda_{(i,l_1)}$ or $\lambda_{(i,l_2),(i,l_3)}^{[a]}$. By Definition \ref{2PO}, $\lambda\mla{\mla{\triangleleft}} \mu$ implies $(i,l_1)\leq (i,l)$ in the first case and $(i,l_2)\leq (i,l)$ in the second. Comparing the two orders in Definition $\ref{2po}$, we have $(i,l_1)\leq' (i,l)$ in the first case and $(i,l_2)\leq' (i,l)$ in the second. As an immediate consequence, $\lambda\mla{\mla{\triangleleft'}} \mu $ in both cases.
         
         (b)  $\mu$ has two non-empty components and is of the form $\mu=\mu_{(i,l_1),(i,l_2)}^{[a]}$. Using the same argument as in the first case, we have $\lambda=\lambda_{(i,l_3),(i,l_4)}^{[b]}$. As the first terms in all the indices are the same, the two partial orders on the indices, $\leq $ and $\leq'$, are equivalent. Therefore, $\lambda_{(i,l_3),(i,l_4)}^{[b]}\mla{\mla{\triangleleft}} \mu_{(i,l_1),(i,l_2)}^{[a]}$ implies $\lambda_{(i,l_3),(i,l_4)}^{[b]}\mla{\mla{\triangleleft'}} \mu_{(i,l_1),(i,l_2)}^{[a]}$. 
         
         (c) $\mu$ is of the form $\mu=\mu_{(i_1,l_1),(i_2,l_2)}^{[a]}$. There are $\dfrac{1}{2}(n+a)$ terms of the form $(i_1,x)$ and $\dfrac{1}{2}(n-a)$ terms of the form $(i_2,y)$ in the residue of $s$. As $res^{\Lambda}(t)=res^{\Lambda}(s)$,the number of nodes in $\lambda$ with residue of the form $(i_1,x)$ is $\dfrac{1}{2}(n+a)$ and that of the form $(i_2,y)$ is $\dfrac{1}{2}(n-a)$. 
         The restriction on the parameters imposed in (\ref{uures}) guarantees that the nodes in the first two rows of a 3D multipartition have different residues. In other words, the nodes in the first two rows are uniquely determined by their residues. As $s$ and $t$ are standard tableau, the number $1$ is in the first row of both $s$ and $t$. Let $c$ be the smallest number such that $Res_s(c)=(i_2,l_2)$.
         Further, $res^{\Lambda}(t)=res^{\Lambda}(s)$ implies that $Res_{t}(1)=Res_{s}(1)=(i_1,l_1)$. Therefore, $t(1)=s(1)=(1,1,(i_1,l_1))$ where $t(1)$ is the node labelled with $1$ in $t$. Similarly, $t(c)=s(c)=(1,1,(i_2,l_2))$. As $\lambda$ consists of at most two non-empty components, we have $\lambda=\lambda_{(i_1,l_1),(i_2,l_2)}^{[a]}=\mu$ which contradicts $\lambda\mla{\mla{\triangleleft}} \mu$. So there are no standard tableaux $s$ and $t$, such that $e(t)=e(s)$ in this case.
         
         Therefore, If there exists a standard tableau $t$ of shape $\lambda$ and a standard tableau $s$ of shape $\mu$ such that
     	\begin{equation*}
     		e(t)=e(s),
     	\end{equation*}
        the 3D multipartition $\mu$ is of the form either $\mu_{(i,l)}$ or $\mu_{(i,l_1),(i,l_2)}^{[a]}$. In both of these cases, $\lambda\mla{\mla{\triangleleft'}} \mu \in \mathfrak{P}_n$ where 
        $\mla{\mla{\triangleleft'}}$ is the finer partial order in Definition \ref{2PO}.
     \end{proof}
     %\ml{Define Garnir tableau}
     
     In the following calculation, we will need to deal with some non-standard tableaux. Following Lobos and Ryom-Hansen, we use Garnir tableaux as a tool to deal with these non-standard tableaux. 
     This method is originally due to Murphy in \cite{murphy1995representations}. We start with the definition of a Garnir tableau.
 
 \begin{defn}
 	Let $\lambda \in \mathfrak{P}_n$ be a 3D multipartition of $n$ and $g$ be a $\lambda$-tableau. We call $g$ a Garnir tableau if and only if there is a unique $k,\;\;1\leq k\leq n-1$ such that $g(k)>g(k+1)$ with respect to the order on nodes defined after $(\ref{5.3})$ and this number $k$ is in the same column as $k+1$.
 	
 \end{defn}
It should be remarked that a  Garnir tableau $g$ is not standard but $g\circ s_k$ is, where $g\circ s_k$ is the tableau obtained by swapping $k$ and $k+1$ in $g$. Because a 3D multipartition in $\mathfrak{P}_n$ consists
of at most two components, either of which has a single column, the rules $(ii)$ and $(iii)$ after $(\ref{5.3})$ define a total order on the places on the floor, the poset $(\mathfrak{P}_n,\unlhd)$ is a subset of poset of the one-column $r$-multipartitions with partial order $\unlhd_0$ defined in \cite{LOBOSMATURANA2020106277}. Therefore, we have:

\begin{lem} \cite[Corollary 22]{LOBOSMATURANA2020106277}$\label{41}$
 	Let $t$ be a non-standard tableau of shape $\lambda\in \mathfrak{P}_n$. Then there exists a Garnir tableau $g$ and $w\in \mathfrak{S}_n$ such that $t=g\circ w$ and $l(d(t))=l(d(g))+l(w)$.
 \end{lem}
 
 %  \ml{end of definition}
 
      \begin{lem} $\label{garlm}$
 	Let $\lambda\mla{\mla{\triangleleft}} \mu \in \mathfrak{P}_n$ be two 3D multipartitions. We have $\lambda\mla{\mla{\triangleleft'}} \mu$
	 if there exists a standard tableau $t$ of shape $\lambda$ and a tableau $g$ of shape $\mu$, which is either standard or Garnir, such that
 	\begin{equation*}
 		e(t)=e(g).
 	\end{equation*}
 \end{lem}
 
 \begin{proof}
 	The proof of cases (a) and (b) defined in the proof of Lemma \ref{weaklm} remains unchanged. 
    
	We claim that the third case (c) does not occur.  If $\mu$ is of form $\mu=\mu_{(i_1,l_1),(i_2,l_2)}^{[a]}$, there are $\dfrac{1}{2}(n+a)$ terms of the form $(i_1,x)$ and $\dfrac{1}{2}(n-a)$ terms of the form $(i_2,y)$ in the residue of $t^{g}$. Let $i$ be such that $Res_g(1)=(i,j)$ and $c$ be the smallest number such that $Res_g(c)=(i',j')$ where $i'\neq i$. As $g$ is a Garnir tableau, $s_k\circ g$ is standard for some $1\leq k\leq n-1$. {If $k=1$ or $c$, $k$ is in the second row and $k+1$ is in the first row of the same component. Otherwise, both $1$ and $c$ are in the first row of $g$. Therefore, $1$ and $c$ are always in the first two row of $g$}
As $res^{\Lambda}(t)=res^{\Lambda}(g)$, we have $Res_t(1)=Res_g(1)$ and $Res_t(c)=Res_g(c)$. Both $1$ and $c$ are in the first row of $t$ as $t$ is standard. As the nodes in the first two rows are uniquely determined by their residue, the nodes containing $1$ and $c$ are in the same position in $t$ and $g$. So we have $\lambda=\lambda_{(i_1,l_1),(i_2,l_2)}^{[a]}=\mu$, which  contradicts $\lambda\mla{\mla{\triangleleft}} \mu$. So there is no standard tableau $t$ of shape $\lambda\mla{\mla{ \triangleleft}} \mu $ such that {$e(t)=e(g)$} for some Garnir tableau $g$ of shape $\mu$.
 \end{proof}

%{We need a little preparation work here to make the following self contained, rather than relying on results from \cite{LL22}}

 Before starting the main theorem, we introduce a notation that will be used in the calculation. Denote $N=\{1,2,\dots,n-1\}$. For $U^{(k)}=(U^{(k)}_1,U^{(k)}_2,\dots,U^{(k)}_k)\in N^k$ where $k$ is a non-negative integer, let $\psi_{U^{(k)}}=\psi_{U^{(k)}_1}\psi_{U^{(k)}_2}\dots \psi_{U^{(k)}_k}$ with $\psi_i(i\in N)$ being the KLR generators in Definition $\ref{klrdf}$. We use the following equivalence relation to describe the difference between two elements in the cyclotomic KLR algebra.

  \begin{defn} \label{dfequiv}
  	For any positive integer $k$, define an equivalence relation $\stackrel{k}{\sim}$ on $N^k$ as follows: 
  We write $V^{(k)}\stackrel{k}{\sim}U^{(k)}$ if 
  	\begin{equation*}
  		\psi_{V^{(k)}}^*e_{\lambda}-\psi_{U^{(k)}}^*e_{\lambda}=\sum_{l<k} c(W^{(l)})\psi_{W^{(l)}}^*e_{\lambda}
  	\end{equation*}
     where $c(W^{(l)})\in R$ and $W^{(l)}$ runs over $N^l$ for all $0\leq l<k$.
  \end{defn}
  The following lemma shows that two reduced expressions of the same element lead to equivalent sequences.
\begin{lem} $\label{45}$
	For $w\in \mathfrak{S}_n$, let $s_{V^{(k)}}$ and $s_{U^{(k)}}$ be two reduced expressions of $w$ where $k=l(w)$, then $V^{(k)}\stackrel{k}{\sim}U^{(k)}$.
\end{lem}
\begin{proof}
	As both $s_{V^{(k)}}$ and $s_{U^{(k)}}$ are reduced expressions of the same element $w$, they can be transformed to each other by braid relations. These relations correspond to $(\ref{22eq})$ and $(\ref{26eq})$. The error terms only occur in $(\ref{26eq})$. They lead to strictly shorter sequences. By Definition \ref{dfequiv}, we have $V^{(k)}\stackrel{k}{\sim}U^{(k)}$.
\end{proof}
%\ml{end}

    For $\lambda\in\mathfrak{P}_n$ and $s,t\in Std(\lambda)$, let $C_{s,t}^{\lambda}\in TL_{r,1,n}(q,\zeta)$ be the element defined in $(\ref{eq104})$. 
    Then by Theorem \ref{cb419}, $\{C_{s,t}^{\lambda}| \lambda\in\mathfrak{P}_n,s,t\in Std(\lambda)\}$ is a graded cellular basis of $TL_{r,1,n}(q,\zeta)$ with respect to the partial order $\trianglelefteq $. We next show that this cellularity still holds with respect to the finer partial order $\trianglelefteq' $, that is,
      
      \begin{thm}\label{8.2.5}
      Let $TL_{r,1,n}(q,\zeta)$ be the specialisation of $TL_{r,1,n}$ in section \ref{8.1.1}.
      	Let $(\mathfrak{P}_n,\trianglelefteq')$ be the poset defined in Definition \ref{2PO} and $C_{s,t}^{\lambda}$ be the element defined in $(\ref{eq104})$. Then $\{C_{s,t}^{\lambda}| \lambda\in\mathfrak{P}_n,s,t\in Std(\lambda)\}$ is a graded cellular basis of $TL_{r,1,n}(q,\zeta)$ over $R$ with respect to with respect to the cell datum $(\mathfrak{P}_n, Std, C,deg)$, {where $\mathfrak{P}_n$ is the poset endowed with 
	the partial order $\trianglelefteq'$}.
      \end{thm}
      \begin{proof}
      	Comparing with the cellular structure in Theorem \ref{cb419}, it is enough to show that for any $a\in TL_{r,1,n}(q,\zeta)$, $\lambda\in \mathfrak{P}_n$ and $s,t\in Std (\lambda)$, we have
      	\begin{equation}
      		aC_{s,t}^{\lambda}=\sum_{s'\in Std(\lambda)}r_a(s',s)C_{s',t}^{\lambda}+\sum_{\mu \triangleleft' \lambda,u,v\in Std(\mu)}c_a(s,t,u,v)C_{u,v}^{\mu}
      	\end{equation}
      	where $r_a(s',s)\in R$ does not depend on $t$ and $c_a(s,t,u,v)\in R$. 
      	We only need to check the cases where $a$ is in the generating set of $TL_{r,1,n}(q,\zeta)$ as a KLR algebra.
      	
      	If $a=e(\mathbf{i})$, $	aC_{s,t}^{\lambda}$ is either 0 or $C_{s,t}^{\lambda}$ by $(\ref{224eq})$ and $(\ref{sym}).$
      	
      		If $a=y_k$, the following equation is obtained by swapping $y_i$ and $\psi_r$ finitely many times using $(\ref{23eq})$ and $(\ref{24eq})$.
      	      	\begin{equation} 
      		y_kC_{s,t}^{\lambda}=y_i\psi_{d(s)}^*e_{\lambda}\psi_{d(t)}= \psi_{d(s)}^*y_je_{\lambda}\psi_{d(t)}+\sum_{l(V)<l({d(s)})}c(V)\psi_{V}^*e_{\lambda}\psi_{d(t)}
      	\end{equation}
      	where $c(V)$ is 1 or -1. 

       % \ml{The following proof is rewritten. The equations are proved directly instead of invoking lemmas from \cite{LL22}}

        For the second part, let $U^{(k)}$ be the sequence of length $k$ such that $\psi_{V}^*=\psi_{U^{(k)}}^*$. We prove the following equation by induction on $k$:
         \begin{equation} \label{89}
      	 	\psi_{U^{(k)}}^*e_{\lambda}\psi_{d(t)}=\sum_{s'\in Std(\lambda)}r_a(s',s)C_{s',t}^{\lambda}+\sum_{i=1}^{n} Y_i+\sum_{g\in Gar(\lambda)} D_g
      	 \end{equation}
     	where $Y_i$ is of the from $\sum_{U_1^{(k_1)},U_2^{(k_2)}}\psi_{U_1^{(k_1)}}^*y_ie_{\lambda}\psi_{U_2^{(k_2)}}$, and 
	
	$D_g$ is of the from $\sum_{U_3^{(k_3)},U_4^{(k_4)}}\psi_{U_3^{(k_3)}}^*e(g)\psi_{U_4^{(k_4)}}$ 
	with $g$ running over the Garnir tableaux of shape $\lambda$. We will omit the superscripts of $U_1$ to $U_4$ in the following proof since we are not interested in the length of these sequences.
        
        If $k=0$, the statement is evident.
        Assume by way of induction, that the assertion is true for any sequence $U^{(l)}$ with $l<k$ for a positive integer $k$. For a sequence $U^{(k)}$, we consider three cases:
        
	1. Assume $s_{U^{(k)}}$ is a reduced expression of some $w\in \mathfrak{S}_n$ and $t^{\lambda}\circ w$ is a standard tableau. Denote by $s$ the standard tableau $t^{\lambda}\circ w$. Let $V^{(k)}$ be the sequence such that $	\psi_{V^{(k)}}^*$ is the chosen $\psi_{d(s)}^*$ in the definition of the cellular basis. Then $s_{V^{(k)}}$ and $s_{U^{(k)}}$ are two reduced expressions of $w$ with $k=l(w)$, so we have $V^{(k)}\stackrel{k}{\sim}U^{(k)}$ by Lemma $\ref{45}$. Thus,
	\begin{equation}
		\psi_{U^{(k)}}^*e_{\lambda}=\psi_{d(s)}^*e_{\lambda}+\sum_{l<k} c(W^{(l)})\psi_{W^{(l)}}^*e_{\lambda}.
	\end{equation}
    By the inductive assumption, $\psi_{W^{(l)}}^*e_{\lambda}\psi_{d(t)}$ is of the form in $(\ref{89})$.
    Therefore, $\psi_{U^{(k)}}^*e_{\lambda}\psi_{d(t)}$ can be transformed into the form in $(\ref{89})$.
    
    2. Assume $s_{U^{(k)}}$ is a reduced expression of some $w\in \mathfrak{S}_n$, but $t^{\lambda}\circ w$ is not a standard tableau. By Lemma $\ref{41}$, there exists a Garnir tableau $g$ and an element $w_1\in \mathfrak{S}_n$ such that $t^{\lambda}\circ w=g\circ w_1$. Let $V(g)$ and $V(w_1)$ be two sequences consisting of numbers in $\{1,2,\dots,n-1 \} $ such that $s_{V(g)}$ and $s_V(w_1)$ are reduced expressions of $d(g)$ and $w_1$ respectively. Let $V^{(k)}=V(g)V(w_1)$ be the combination of these two sequences. Then $s_{V^{(k)}}$ is a reduced expression of $w$ as well. By Lemma $\ref{45}$, $V^{(k)}\stackrel{k}{\sim}U^{(k)}$. So we have
    \begin{equation*}
    \begin{aligned}
    	\psi_{U^{(k)}}^*e_{\lambda}
    	&=\psi_{V^{(k)}}^*e_{\lambda}+\sum_{l<k} c(W^{(l)})\psi_{W^{(l)}}^*e_{\lambda}\\
    	&=\psi_{V(w_1)}^*\psi_{V(g)}^*e_{\lambda}+\sum_{l<k} c(W^{(l)})\psi_{W^{(l)}}^*e_{\lambda}\\
    	&=\psi_{V(w_1)}^*e(g)\psi_{V(g)}^*+\sum_{l<k} c(W^{(l)})\psi_{W^{(l)}}^*e_{\lambda}	.
    \end{aligned}        
    \end{equation*}

    The first term is of the from $\sum_{U_3,U_4}\psi_{U_3}^*e(g)\psi_{U_4}$ and the other terms satisfy that the length of the sequence decreases strictly, so the inductive assumption can be used.

    3. If $s_{U^{(k)}}$ is not a reduced expression of any $w\in \mathfrak{S}_n$, let $l$ be the largest number such that $s_{U^{(k)}|_l}$ is a reduced expression of some $w\in \mathfrak{S}_n$ where $U^{(k)}|_l$ is the sub-sequence of $U^{(k)}$ consisting of the first $l$ terms. By the exchange condition, there exists a sequence $V^{(l)}$ ending with $U^{(k)}_{l+1}$ such that $s_{V^{(l)}}$ is a reduced expression for $s_{U^{(k)}|_l}$. By Lemma $\ref{45}$, $V^{(l)}\stackrel{k}{\sim}U^{(k)}|_l$. Therefore, we have
    \begin{equation*}
       \begin{aligned}
    	\psi_{U^{(k)}}^*e_{\lambda}&=(\psi_{U^{(k)}|_l}\psi_{U^{(k)}_{l+1}}\psi_{U^{(k-l-1)}})^*e_{\lambda}\\
    	&=(\psi_{V^{(l)}}\psi_{U^{(k)}_{l+1}}\psi_{U^{(k-l-1)}})^*e_{\lambda}+\sum_{l<k} c(W^{(l)})\psi_{W^{(l)}}^*e_{\lambda}\\
    	&=(\psi_{V^{(l)}|_{l-1}}\psi_{U^{(k)}_{l+1}}^2\psi_{U^{(k-l-1)}})^*e_{\lambda}+\sum_{l<k} c(W^{(l)})\psi_{W^{(l)}}^*e_{\lambda}\\
    	&=\psi_{U^{(k-l-1)}}^*e(i)\psi_{U^{(k)}_{l+1}}^2\psi_{V^{(l)}|_{l-1}}^*+\sum_{l<k} c(W^{(l)})\psi_{W^{(l)}}^*e_{\lambda}	.
    \end{aligned}     
    \end{equation*}

    For the same reason as above, we only need to deal with the first term. By ($\ref{25eq}$), $e(i)\psi_{U^{(k)}_{l+1}}^2$ is $0$, $e(i)$, $	(y_{s+1}-y_s)e(i) $ or 	$(y_s-y_{s+1})e(i) $. The first case is trivial. The second case leads to a strictly shorter sequence, so the inductive assumption can be used. If we get $	(y_{s+1}-y_s)e(i) $ or 	$(y_s-y_{s+1})e(i) $, ($\ref{23eq}$) can help to move $y_s$ to the right to get some elements $Y_i$. The error term will also lead to a strictly shorter sequence which is covered by the inductive assumption.

    If $a=\psi_l$, we have
    \begin{equation}
        aC_{s,t}^{\lambda}=\psi_l \psi_{d(s)}^*e_{\lambda}\psi_{d(t)}.
    \end{equation}
    Let $U^{(k)}$ be the sequence of length $k$ such that $\psi_{U^{(k)}}^*=\psi_l \psi_{d(s)}^*$, we can use the same induction on $k$ as above to prove:
     	\begin{equation} 
     		aC_{s,t}^{\lambda}=\sum_{s'\in Std(\lambda)}r_a(s',s)C_{s',t}^{\lambda}+\sum_{i=1}^{n} Y_i+\sum_{g\in Gar(\lambda)} D_g.
      	\end{equation}
        where $Y_i$ is of the from $\sum_{U_1,U_2}\psi_{U_1}^*y_ie_{\lambda}\psi_{U_2}$ and $D_g$ is of the from $\sum_{U_3,U_4}\psi_{U_3}^*e(g)\psi_{U_4}$ with $g$ running over the Garnir tableaux of shape $\lambda$.
        
%\ml{end of the proof of the equations}
        
 %       The first term is in the two-sided ideal generated by $y_je_{\lambda}$. The other terms can be transformed by Lemma 6.7 in \cite{LL22} as follows:
      	% \begin{equation} 
      	% 	\psi_{V}^*e_{\lambda}\psi_{d(t)}=\sum_{s'\in Std(\lambda)}r_a(s',s)C_{s',t}^{\lambda}+\sum_{i=1}^{n} Y_i+\sum_{g\in Gar(\lambda)} D_g
      	% \end{equation}
    %  	where $Y_i$ is in the two-sided ideal generated by $y_ie_{\lambda}$ and $D_g$ is in the one generated by $e(g)$ with $g$ running over the Garnir tableaux of shape $\lambda$.
      	
   %    If $a=\psi_k$, Lemma  6.7 in \cite{LL22}  indicates that
    %  	\begin{equation} 
   %   		aC_{s,t}^{\lambda}=\sum_{s'\in Std(\lambda)}r_a(s',s)C_{s',t}^{\lambda}+\sum_{i=1}^{n} Y_i+\sum_{g\in Gar(\lambda)} D_g
  %    	\end{equation}
    %  	where $Y_i$ is in the two-sided ideal generated by $y_ie_{\lambda}$ and $D_g$ is in the one generated by $e(g)$ with $g$ running over the Garnir tableaux of shape $\lambda$.
      	
      	We next show that 
      	      	\begin{equation}
      		y_ie_{\lambda}=\sum_{\mu \triangleleft' \lambda,u,v\in Std(\mu)}c_i(u,v)C_{u,v}^{\mu}
      	\end{equation}
      for $1\leq i\leq n$, and
            	\begin{equation}
      	e(g)=\sum_{\mu \triangleleft' \lambda,u,v\in Std(\mu)}c_g(u,v)C_{u,v}^{\mu}
      \end{equation}
      for $g\in Gar(\lambda)$.
      We treat $y_ie_{\lambda}$ first. By Lemma 17 and Lemma 37 in \cite{LOBOSMATURANA2020106277}, we have
      	\begin{equation}
      	y_ie_{\lambda}=\sum_{\mu \triangleleft \lambda,u,v\in Std(\mu)}c_i(u,v)C_{u,v}^{\mu}.
      \end{equation}
     As $y_ie_{\lambda}=e_{\lambda}y_i$ and $e_{\lambda}$ is an idempotent, we have
      \begin{align*}
      	y_ie_{\lambda}&=e_{\lambda}y_ie_{\lambda}\\
      	    &=\sum_{\mu \triangleleft \lambda,u,v\in Std(\mu)}c_i(u,v)e_{\lambda}C_{u,v}^{\mu}\\
      	    &=\sum_{\mu \triangleleft \lambda,u,v\in Std(\mu)}c_i(u,v)e_{\lambda}\psi_{d(u)}^*e_{\mu}\psi_{d(v)}\\
      	    &=\sum_{\mu \triangleleft \lambda,u,v\in Std(\mu)}c_i(u,v)e_{\lambda}e(u)\psi_{d(u)}^*\psi_{d(v)}\\
      	    &=\sum_{\mu \triangleleft \lambda,u,v\in Std(\mu),e_{\lambda}=e(u)}c_i(u,v)e(u)\psi_{d(u)}^*\psi_{d(v)}\\
      	    &=\sum_{\mu \triangleleft \lambda,u,v\in Std(\mu),e_{\lambda}=e(u)}c_i(u,v)C_{u,v}^{\mu}\\
      	    &=\sum_{\mu \triangleleft' \lambda,u,v\in Std(\mu)}c_i(u,v)C_{u,v}^{\mu}
      \end{align*}
      where the last step is by Lemma $\ref{weaklm}$ and the coefficients $c_i(u,v)$ are zero if $e_{\lambda}\neq e(u)$ in the last summand.
      
      For $e(g)$, Lemma 35 in \cite{LOBOSMATURANA2020106277} implies that
      \begin{equation*}
      	e(g)=\sum_{\mu \triangleleft \lambda,u,v\in Std(\mu)}c_g(u,v)C_{u,v}^{\mu}.
      \end{equation*}
      As $e(g)$ is an idempotent, we have
      \begin{align*}
      	e(g)&=e(g)^2\\
      	    &=e(g)\sum_{\mu \triangleleft \lambda,u,v\in Std(\mu)}c_g(u,v)C_{u,v}^{\mu}\\
      	    &=\sum_{\mu \triangleleft \lambda,u,v\in Std(\mu)}c_g(u,v)e(g)\psi_{d(u)}^*e_{\mu}\psi_{d(v)}\\
      	    &=\sum_{\mu \triangleleft \lambda,u,v\in Std(\mu)}c_g(u,v)e(g)e(u)\psi_{d(u)}^*\psi_{d(v)}\\
      	    &=\sum_{\mu \triangleleft \lambda,u,v\in Std(\mu),e(u)=e(g)}c_g(u,v)e(u)\psi_{d(u)}^*\psi_{d(v)}\\
      	    &=\sum_{\mu \triangleleft \lambda,u,v\in Std(\mu),e(u)=e(g)}c_g(u,v)C_{u,v}^{\mu}\\
      	    &=\sum_{\mu \triangleleft' \lambda,u,v\in Std(\mu)}c_g(u,v)C_{u,v}^{\mu}\\
      \end{align*}
        where the last step is by Lemma $\ref{garlm}$ and the coefficients $c_g(u,v)$ are zero if $e_{g}\neq e(u)$ in the last sum.

       % \ml{The following part is to correct the mistake pointed out by the third last main correction from the referees}
        
We next show that for any sequences $U_1, U_2\in N^{\infty}$ of finite length, $\mu\in \mathfrak{P}_n$ and $s,t\in Std (\mu)$, we have
\begin{equation*}
    \psi_{U_1}^*C_{s,t}^{\mu}\psi_{U_2} =\sum_{\nu \trianglelefteq' \mu,u,v\in Std(\nu)}c(U_1,U_2,s,t,u,v)C_{u,v}^{\nu}.
\end{equation*}
We have
\begin{align*}
      	\psi_{U_1}^*C_{s,t}^{\mu}
      	    &=\psi_{U_1}^*\psi_{d(s)}^*e_{\mu}\psi_{d(t)} \\
      	    &=\psi_{U_1}^*\psi_{d(s)}^*\psi_{d(t)}e(t)\\
            &=\psi_{U_1}^*\psi_{d(s)}^*\psi_{d(t)} e(t)^2\\
            &=\psi_{U_1}^*C_{s,t}^{\mu}e(t)\\
            &=\sum_{\nu \trianglelefteq \mu,u,v\in Std(\nu)}c(U_1,s,t,u,v)C_{u,v}^{\nu}e(t)\\
            &=\sum_{\nu \trianglelefteq \mu,u,v\in Std(\nu)}c(U_1,s,t,u,v)\psi_{d(u)}^*\psi_{d(v)}e(v)e(t).
      \end{align*}
The second last step is from the cellular structure associated with the partial order $\trianglelefteq$. Since $t$ is a standard tableau, Lemma $\ref{weaklm}$ implies:
\begin{align*}
      	\psi_{U_1}^*C_{s,t}^{\mu} 
            &=\sum_{\nu \trianglelefteq \mu,u,v\in Std(\nu)}c(U_1,s,t,u,v)\psi_{d(u)}^*\psi_{d(v)}e(v)e(t)\\
            &= \sum_{\nu \trianglelefteq \mu,u,v\in Std(\mu),e(v)=e(t)}c(U_1,s,t,u,v)\psi_{d(u)}^*\psi_{d(v)}e(v)e(t)\\
            &= \sum_{\nu \trianglelefteq \mu,u,v\in Std(\mu),e(v)=e(t)}c(U_1,s,t,u,v)\psi_{d(u)}^*\psi_{d(v)}e(v)\\
            &=\sum_{\nu \trianglelefteq' \mu,u,v\in Std(\nu)}c(U_1,s,t,u,v)C_{u,v}^{\nu}.
      \end{align*}
Using a similar argument, we have
\begin{equation*}
      	C_{s,t}^{\mu} \psi_{U_2}
            =\sum_{\nu \trianglelefteq' \mu,u,v\in Std(\nu)}c(U_2,s,t,u,v)C_{u,v}^{\nu}.
      \end{equation*}
Therefore, we have
\begin{align*}
      	\psi_{U_1}^*C_{s,t}^{\mu} \psi_{U_2}
            &=\sum_{\tau \trianglelefteq' \mu,x,y\in Std(\tau)}c(U_1,s,t,x,y)C_{x,y}^{\tau}\psi_{U_2}\\
            &=\sum_{\tau \trianglelefteq' \mu,x,y\in Std(\tau)}\sum_{\nu \trianglelefteq' \tau,u,v\in Std(\nu)}c(U_1,U_2,s,t,x,y,u,v)C_{u,v}^{\nu}\\
            &=\sum_{\nu \trianglelefteq' \mu,u,v\in Std(\nu)}c(U_1,U_2,s,t,u,v)C_{u,v}^{\nu}.            
      \end{align*}
      Thus the terms $T_i$ and $D_g$ are all lower terms with respect to the finer partial order $\trianglelefteq'$.
      
    %    \ml{correction ends here}
        
        Therefore, for any $a\in TL_{r,1,n}(q,\zeta)$, $\lambda\in \mathfrak{P}_n$ and $s,t\in Std (\lambda)$, we have
        \begin{equation}
        	aC_{s,t}^{\lambda}=\sum_{s'\in Std(\lambda)}r_a(s',s)C_{s',t}^{\lambda}+\sum_{\mu \triangleleft' \lambda,u,v\in Std(\mu)}c_a(s,t,u,v)C_{u,v}^{\mu}.
        \end{equation}
      Therefore, $\{C_{s,t}^{\lambda}| \lambda\in\mathfrak{P}_n,s,t\in Std(\lambda)\}$ is a graded cellular basis of $TL_{r,1,n}(q,\zeta)$ over $R$ with respect to $\trianglelefteq'$.
      \end{proof}
      
      Now consider the corresponding poset automorphism $\sigma_{\mathfrak{P}_n}$ on the poset $(\mathfrak{P}_n,\trianglelefteq')$. Let $\sigma_{\mathfrak{P}_n}:\mathfrak{P}_n\mapsto\mathfrak{P}_n$ be the bijective map given by:
      \begin{align*}
      	\sigma_{\mathfrak{P}_n}(\lambda_{(i,l)})&=\lambda_{(i-1,l)},\\
      	\sigma_{\mathfrak{P}_n}(\lambda_{(i_1,l_1),(i_2,l_2)}^{[a]})&=\lambda_{(i_1-1,l_1),(i_2-1,l_2)}^{[a]},
      \end{align*}
  	  where $i-1=p-1$ when $i=0$. In other words, $\sigma_{\mathfrak{P}_n}$ shifts the 3D-multipartitions forward by one column. We may now prove:
  	  \begin{lem}
  	  	$\sigma_{\mathfrak{P}_n}$ is a poset automorphism of $(\mathfrak{P}_n,\trianglelefteq')$.
  	  \end{lem}
    \begin{proof}
    	The bijectivity is obvious. We prove that $\sigma_{\mathfrak{P}_n}$ preserves the partial order $\trianglelefteq'$. For any $(i_1,l_1)<'(i_2,l_2)\in \mathcal{K}$, we have  $l_1<l_2$ which implies that $(i_1-1,l_1)<'(i_2-1,l_2)$. Therefore, $(i_1,l_1)\leq'(i_2,l_2)$ implies $(i_1-1,l_1)\leq'(i_2-1,l_2)$. By Definition $\ref{2PO}$, $\lambda \mla{\mla{\triangleleft'}}\mu\in \mathfrak{P}_n$ implies $\sigma_{\mathfrak{P}_n}(\lambda )\mla{\mla{\triangleleft'}}\sigma_{\mathfrak{P}_n}(\mu)$. Therefore, $\sigma_{\mathfrak{P}_n}$ is a poset automorphism of $(\mathfrak{P}_n,\trianglelefteq')$. 
    \end{proof}
     For $t\in Std(\lambda)$ where $\lambda\in\mathfrak{P}_n$, define $\sigma_{\mathcal{T}}(t)$ as the standard tableau with shape $\sigma_{\mathfrak{P}_n}(\lambda)$ and all the numbers in the same relative positions as those in $t$.
     \begin{thm}\label{thshift}
     	Let $\sigma_{TL}$ be the automorphism of $TL_{r,1,n}(q,\zeta)$ defined by $(\ref{eq817})$, $\sigma_{\mathfrak{P}_n}$ and $\sigma_{\mathcal{T}}$ be the corresponding automorphisms described above. Then $\sigma=(\sigma_{TL}, \sigma_{\mathfrak{P}_n},\sigma_{\mathcal{T}})$ is a shift automorphism ($cf.$ Definition $\ref{shift}$) with respect to the cell datum $(\mathfrak{P}_n, Std, C,deg)$ of $TL_{r,1,n}(q,\zeta)$ where $\mathfrak{P}_n$ is the poset with $\trianglelefteq'$ (cf. Theorem \ref{8.2.5}).
     \end{thm}
     \begin{proof}
     	By the argument above, $\sigma_{TL}$, $\sigma_{\mathfrak{P}_n}$ and $\sigma_{\mathcal{T}}$ are well defined.
 %        The conditions ($a$) and ($b$) in Definition $\ref{shift}$ can be checked directly.
 
 We now check the conditions (a), (b) and (c) of Definition \ref{shift} explicitly.
According to the definition, $\sigma_{\mathcal{T}}$ translates the numbers without changing the relative positions in the tableau. More precisely, the tableaux are moved one space to the left on the floor by $\sigma_{\mathcal{T}}$(cf. Figure$\ref{3dir}$ and Figure$\ref{3dr}$). Therefore, for any $t\in Std(\lambda)$, $\sigma_{\mathcal{T}}(t)\in Std(\sigma_{\mathfrak{P}_n}(\lambda))$ and $deg(\sigma_{\mathcal{T}}(t))=deg(t)$, thus (a) holds.

For $s,t\in Std(\lambda)$, we have
\begin{equation*}
C_{s,t}^{\lambda}=\psi_{d(s)}^*e_{\lambda}\psi_{d(t)},    \end{equation*}
and
\begin{equation*}
    C_{\sigma_\mathcal{T}(t),\sigma_\mathcal{T}(s)}^{\sigma_{\mathfrak{P}_n}(\lambda)}=\psi_{d(\sigma_\mathcal{T}(s))}^*e_{\sigma_{\mathfrak{P}_n}(\lambda)}\psi_{d(\sigma_\mathcal{T}(t))}.
\end{equation*}
Since  $\sigma_{\mathcal{T}}$ doesn't change the relative positions in the tableau, we have $\psi_{d(s)}=\psi_{d(\sigma_\mathcal{T}(s))}$ and $\psi_{d(t)}=\psi_{d(\sigma_\mathcal{T}(t))}$.
Since $\sigma_{\mathfrak{P}_n}(\lambda)$ is obtained by moving $\lambda$ one space to the left on the floor, if $\mathbf{i}=((i_1,l_1),(i_2,l_2),\dots, (i_n,l_n))\in \mathcal{K}^n$ such that $e_{\lambda}=e(\mathbf{i})$, $e_{\sigma_{\mathfrak{P}_n}(\lambda)}=e((i_1-1,l_1),(i_2-1,l_2),\dots, (i_n-1,l_n))=\sigma_{TL}(e_{\lambda})$, where $(-1,l)=(p-1,l)$. Therefore, we have
\begin{equation*}
    C_{\sigma_\mathcal{T}(t),\sigma_\mathcal{T}(s)}^{\sigma_{\mathfrak{P}_n}(\lambda)}=\psi_{d(s)}^*\sigma_{TL}(e_{\lambda})\psi_{d(t)}=\sigma_{TL}(C_{s,t}^{\lambda}),
\end{equation*}
thus condition (b) holds.

        For any $t\in Std(\lambda)$, since the tableaux are moved one space to the left on the floor by $\sigma_{\mathcal{T}}$, $\sigma_{\mathcal{T}}^k(t)=t$ if and only if $p|k$. So (c) holds for all $\lambda\in\mathfrak{P}_n$.
        %  \ml{end of check}
     \end{proof}
     We are now in a position to identify the skew cell datum of the point-wise fixed subalgebra $TL_{r,p,n}$ under the shift automorphism $\sigma$. Let $\equiv$ be the equivalence relation on $\mathfrak{P}_n$ such that
     \begin{equation}\label{equmpar}
           \begin{aligned}
     	\lambda_{(i,l)}&\equiv \lambda_{(i+k,l)}\\
     	\mu_{(i_1,l_1),(i_2,l_2)}^{[a]}&\equiv\mu_{(i_1+k,l_1),(i_2+k,l_2)}^{[a]},
     \end{aligned}
     \end{equation}
   
     for all $k\in I$ and $\lambda_{(i,l)},\mu_{(i_1,l_1),(i_2,l_2)}^{[a]}\in \mathfrak{P}_n$. Notice that $\lambda\equiv\sigma_{\mathfrak{P}_n}(\lambda)$ for all $\lambda \in \mathfrak{P}_n$. Denote by $[\lambda]$ the equivalence class in $\mathfrak{P}_n$ containing $\lambda$. Then $[\lambda]$ is the orbit of $\lambda$ under $\sigma_{\mathfrak{P}_n}$. Define the set of orbits
     \begin{equation}\label{poset831}
     	\mathfrak{P}_{n,p}:=\{[\lambda]|\lambda\in\mathfrak{P}_n \}.
     \end{equation}
     For $[\lambda],[\mu] \in\mathfrak{P}_{n,p}$, write $[\lambda]\trianglelefteq_p[\mu]$ if and only if there exist $\lambda_1\equiv\lambda$ and $\mu_1\equiv\mu$ such that $\lambda_1\trianglelefteq'\mu_1\in\mathfrak{P}_n$. 
     Recall that $[\lambda]\trianglelefteq_p[\mu]$ does not imply $\lambda\trianglelefteq'\mu$, as they may not be comparable with respect to $\trianglelefteq'$. We next provide a counterexample.
     \begin{exmp}
     Let $r=12$, $p=3$ and $n=5$.{The floor has size $3\times 4$.} Then $\lambda_{(1,1),(2,3)}^{[1]}$ and $\mu_{(0,1)}$ are two 3D multipartitions in $\mathfrak{P}_5$. We have
     \begin{equation*}
     	[\lambda_{(1,1),(2,3)}^{[1]}]\trianglelefteq_p[\mu_{(0,1)}]
     \end{equation*}
     because $\lambda_{(1,1),(2,3)}^{[1]}\equiv\lambda_{(0,1),(1,3)}^{[1]}\trianglelefteq'\mu_{(0,1)}$ according to $(2)$ in Definition $\ref{2PO}$. But $\lambda_{(1,1),(2,3)}^{[1]}$ and $\mu_{(0,1)}$ are incomparable with respect to $\trianglelefteq'$.
     \end{exmp}

   {We now prove the crucial result that the involution $\iota_{\sigma}$ in the skew cell datum in Theorem $\ref{skc}$ is trivial by showing that the orbit size $o_T(\lambda)=1$. Since the two maps $\sigma_{\mathfrak{P}_n}$ and $\sigma_{\mathcal{T}}$ move the diagrams and tableaux one place to the left on the floor, for any $\lambda\in\mathfrak{P}_n$ and $t\in Std(\lambda)$, the size of the $\langle \sigma_{\mathfrak{P}_n}\rangle$-orbit containing $\lambda$ and the size of the $\langle\sigma_{\mathcal{T}}\rangle$-orbit containing $t$ are both $p$. Comparing the definition of $\sigma_{\lambda}$ in the general theory expounded before Theorem $\ref{skc}$, we have $\sigma_{\lambda}=\sigma_{\mathcal{T}}^p=id$.
     As a direct consequence, the size of orbits of $\sigma_{\lambda}$ in $T(\lambda)$, $o_T(\lambda)=1$.
     The poset involution $\iota_p$ in Theorem $\ref{skc}$ is defined by $\iota_p((\lambda,k))=(\lambda,-k)$ with $\lambda \in \mathfrak{P}$ and $k\in \mathbb{Z}/o_T(\lambda)\mathbb{Z}$. Since $o_T(\lambda)=1$, $\iota_p=id$ in the skew cell datum of the point-wise fixed subalgebra $TL_{r,p,n}$
     This implies the skew cellular structure we construct is actually cellular.}
     
     In analogy to the equivalence relation among the multipartitions, we define an equivalence relation $\equiv$ on the set of tableaux $\mathcal{T}=\bigsqcup_{\lambda\in \mathfrak{P}}T(\lambda)$ as follows:
     
     For $s,t\in \mathcal{T}$, we say $s\equiv t$ if and only if $shape(s)\equiv shape(t)$ and the numbers in $s$ are in the same relative positions as those in $t$. Denote by $[t]$ the equivalence class containing $t$ in $\mathcal{T}$. For $s\equiv t\in \mathcal{T}$, as the numbers in $s$ are in the same relative positions as those in $t$, we have $d(s)=d(t)$ and $deg(s)=deg(t)$, where $d(s)$ is the element in $\mathfrak{S}_n$ transforming $s$ to the unique standard tableau $t^{\lambda}$. Therefore, the element $d([t])$ can be defined as $d(s)$ and $deg_p([t])$ can be defined as $deg(s)$ for any $s\equiv t$.
     For $[\lambda]\in\mathfrak{P}_{n,p}$, define:
     \begin{equation}\label{t_p}
     	T_p([\lambda]):=\{[t]|t\in Std(\lambda) \}.
     \end{equation}
     Finally, for $[s],[t]\in T_p([\lambda])$, define
     \begin{equation} \label{cbrpn}
     	C_p^{[\lambda]}([s],[t])=\psi_{d([s])}^*(\sum_{\mu\equiv\lambda}e_{\mu})\psi_{d([t])}.
     \end{equation}
     
     The following Theorem is a direct consequence of Theorem $\ref{skc}$:
     \begin{thm}\label{csrpn}
        
     	Let $TL_{r,p,n}$ be the Temperley-Lieb algebra of type $G(r,p,n)$ defined in Definition \ref{dfrpn}, $\mathfrak{P}_{n,p}$ be the poset of equivalence classes of multipartitions defined in (\ref{poset831} and $T_p,C_p,deg_p$ be as described above. Then $TL_{r,p,n}$ is a skew cellular algebra with skew cell datum $(\mathfrak{P}_{n,p}, id,T_p,C_p,deg_p)$. Moreover, since the poset involution is trivial, $TL_{r,p,n}$ is a cellular algebra with cell datum $(\mathfrak{P}_{n,p},T_p,C_p,deg_p)$.
     \end{thm}

\section{The representations of $TL_{r,p,n}$}
In this section, we study the representations of $TL_{r,p,n}$ from the cellular point of view. We first calculate the dimensions of the cell modules and define reducible and irreducible elements in the poset $\mathfrak{P}_{n,p}$. These names come from the fact that the cell modules corresponding to the irreducible elements are simple. Then we show that the cell modules of $TL_{r,p,n}$ corresponding to reducible elements in $\mathfrak{P}_{n,p}$ can be regarded as cell modules of some $TL_{d,1,n}$ where $d=\frac{r}{p}$ with special parameters. Finally, we calculate the decomposition numbers of the reducible cell modules.
\subsection{The cell modules of $TL_{r,p,n}$ and irreducible cells}
In this subsection, we give a description of the cell modules of $TL_{r,p,n}$. Let $[\lambda]\in\mathfrak{P}_{n,p}$ be an equivalence class of multipartitions in $\mathfrak{P}_n$ with respect to the equivalence relation $\equiv$ defined in (\ref{equmpar}). The (left) cell module of $TL_{r,p,n}$ corresponding to $[\lambda]$, denoted $W([\lambda])$, is the free module over $R$ with basis $\{C_{[s]}|[s]\in T([\lambda]) \}$. The $TL_{r,p,n}$-action is defined by
	\begin{equation}
		aC_{[s]}=\sum_{[s']\in T([\lambda])}r_a([s'],[s])C_{[s']}
	\end{equation}
    for all $a\in TL_{r,p,n}$ and $[s]\in T([\lambda])$, where $r_a([s'],[s])$ is the coefficient uniquely defined by 
        \begin{equation}
        	aC_{[s],[t]}^{[\lambda]}=\sum_{[s']\in T([\lambda])}r_a([s'],[s])C_{[s'],[t]}^{[\lambda]}+\sum_{[\mu] \triangleleft_p [\lambda],[u],[v]\in T([\mu])}dC_{[u],[v]}^{[\mu]}.
        \end{equation}
The right cell modules of $TL_{r,p,n}$ are defined similarly. 
We next calculate the dimensions of the cell modules. For $[\lambda]\in \mathfrak{P}_{n,p}$, the dimension of the cell module $W([\lambda])$ equals to the number of standard tableaux of shape $\lambda$. Therefore, 
for $\lambda_{(i,l)},\mu_{(i_1,l_1),(i_2,l_2)}^{[a]} \in \mathfrak{P}_n$, we have
\begin{equation}\label{eq:dimW}
\begin{aligned}
    dim(W([\lambda_{(i,l)}]))=&1\\
    dim(W([\mu_{(i_1,l_1),(i_2,l_2)}^{[a]}]))=&\binom{n}{\frac{n-a}{2}}.
\end{aligned}
\end{equation}

Further, denote by $TL_{\mla{\mla{\triangleleft}}_p[\lambda]}$ and $TL_{\trianglelefteq_p[\lambda]}$ the subspaces spanned by $\{C_{[s],[t]}^{[\mu]}|[\mu]\mla{\mla{\triangleleft}}_p[\lambda]\}$ and $\{C_{[s],[t]}^{[\mu]}|[\mu]\trianglelefteq_p[\lambda]\}$ respectively. These are $TL_{r,p,n}$-bimodules by definition of a cell datum. Define the $TL_{r,p,n}$-bimodule corresponding to $[\lambda]$ as
\begin{equation}\label{eq:filtrn}
    TL([\lambda]):=TL_{\trianglelefteq_p[\lambda]}/TL_{\mla{\mla\triangleleft}_p[\lambda]}.
\end{equation}

Now the argument of \cite[Lemma 2.2]{Lehrer1996} shows that for each $\lambda$, $TL([\lambda])\cong W(\lambda)\otimes W(\lambda)^*$. Thus, from \eqref{eq:dimW},\eqref{eq:filtrn}  we have
\begin{equation}
	dim(TL([\lambda]))=
	\begin{cases}
	    1 & \textit{ if } \lambda=\lambda_{(i,l)};\\
	    \binom{n}{\frac{n-a}{2}}^2 & \textit{ if } \lambda= \lambda_{(i_1,l_1),(i_2,l_2)}^{[a]}.
	\end{cases}
\end{equation}
This permits the calculation of the dimension of $TL_{r,p,n}$:
\begin{align*}
	dim(TL_{r,p,n})&=\sum_{[\lambda]\in\mathfrak{P}_{n,p}} dim(TL([\lambda ]))\\
	             &=\sum_{[\lambda]\in\mathfrak{P}_{n,p}(1)} dim(TL([\lambda ]))+\sum_{[\lambda]\in\mathfrak{P}_{n,p}(2)} dim(TL([\lambda ]))\\
	             &=d+\frac{1}{p}\binom{r}{2}\left(\sum_{a_1=1}^{n-1}\binom{n}{a_1}^2\right)\\
	             &=d+\frac{1}{p}\binom{r}{2}\left(\binom{2n}{n}-2\right)\\
	             &=\frac{1}{p}\left(\binom{r}{2}\binom{2n}{n}-r^2+2r\right)
\end{align*}
where $\mathfrak{P}_{n,p}(i)$ consists of the equivalence classes of 3D multipartitions with exactly $i$ non-empty components.
 By Definition \ref{irrr}, the reducibility of $\lambda$ is equivalent to that of any $\lambda'\in [\lambda]$. So we can say $[\lambda]\in\mathfrak{P}_{n,k}$ is reducible if $\lambda\in\mathfrak{P}_n$ is.
 
{Partition} the poset $\mathfrak{P}_{n,p}$ into two parts, {\it viz.} those subsets $\mathfrak{P}_{n,p}^0$ consisting of the reducible orbits and the $\mathfrak{P}_{n,p}^1$ consisting of the irreducible orbits. We will show that the cell modules corresponding to these two parts have different properties. We begin by showing that the bimodules corresponding to the irreducible orbits are irreducible two-sided ideals of $TL_{r,p,n}$.

\begin{lem} \label{impr}
Let $[\lambda]\in \mathfrak{P}_{n,p}$ be irreducible. For any $[\mu]\neq [\lambda]$, $[s],[t]\in T([\mu])$ and $[u],[v]\in T([\lambda])$, we have
\begin{equation*}
    	C_p^{[\mu]}([s],[t])C_p^{[\lambda]}([u],[v])=0.
\end{equation*}
\end{lem}
\begin{proof}
  By (\ref{cbrpn}), we have
  \begin{align*}
      &C_p^{[\mu]}([s],[t])C_p^{[\lambda]}([u],[v])\\
      =&\psi_{d([s])}^*(\sum_{\mu'\equiv\mu}e_{\mu'})\psi_{d([t])}\psi_{d([u])}^*(\sum_{\lambda'\equiv\lambda}e_{\lambda'})\psi_{d([v])}\\
      =&\psi_{d([s])}^*\psi_{d([t])}(\sum_{t'\equiv t}e(t'))(\sum_{u'\equiv u}e(u'))\psi_{d([u])}^*\psi_{d([v])}.
  \end{align*}
  It suffices to show
  \begin{equation*}
      (\sum_{t'\equiv t}e(t'))(\sum_{u'\equiv u}e(u'))=0.
  \end{equation*}
  As the distinct $e(\mathbf{i})$ are mutually orthogonal, we need only show that $e(t)\neq e(u)$ for all $u\in Std(\lambda)$ and $t\in Std(\mu)$ with $\mu\not\equiv \lambda$. Let $\lambda $ be of the form $\lambda_{(i_1,l_1),(i_2,l_2)}^{[a]}$ with $i_1\neq i_2$ and $e(u)=e(k_1,k_2,\dots,k_n)$. As $u$ is a standard tableau of shape $\lambda$, $\frac{n+a}{2}$ of the $k$'s are of the form $(i_1,x)$ and  $\frac{n-a}{2}$ $k$'s are of the form $(i_2,x)$. If $e(t)= e(u)$, there should be $\frac{n+a}{2}$ nodes in $\mu$ with residues of the form $(i_1,x)$ and $\frac{n-a}{2}$ nodes in $\mu$ with residues of the form $(i_2,x)$. As $\mu$ has at most 2 non-empty components, $\mu=\mu_{(i_1,l_3),(i_2,l_4)}^{[a]}$.
  
  We shall show that $l_3=l_1$ and $l_4=l_2$. Denote by $k_a$ the first $k$ of the form $(i_1,x)$ and $k_b$ the first of the form $(i_2,x)$. As $u$ is a standard tableau, the numbers $a$ and $b$ are in the first row. For the same reason, the numbers $a$ and $b$ are in the first row of $t$. According to the restriction on the parameters in (\ref{uures}), the nodes in the first row are uniquely determined by their residues. Therefore, $a$ and $b$ are in the same position in $u$ and $t$, which implies that $l_3=l_1$ and $l_4=l_2$. Thus $\mu=\lambda$, contradicting our assumption that $\mu\not\equiv \lambda$. So $e(t)\neq e(u)$, so that $e(t)e(u)=0$.
  
  So for any standard tableau $t'\equiv t$, we have $Shape(t')\equiv \mu\not\equiv \lambda$. Thus,
  \begin{equation*}
      \left(\sum_{t'\equiv t}e(t')\right)e(u')=0.
  \end{equation*}
  But for any standard tableau $u'\equiv u$, $Shape(u')\equiv \lambda\not\equiv \mu'$. We have
    \begin{equation*}
      \left(\sum_{t'\equiv t}e(t')\right)\left(\sum_{u'\equiv u}e(u')\right)=0.
  \end{equation*}
  Therefore, $C_p^{[\mu]}([s],[t])C_p^{[\lambda]}([u],[v])=0$.
\end{proof}

By applying the anti-involution of $TL_{r,p,n}$ which fixes the KLR generators, we further obtain:
\begin{lem} \label{832}
Let $[\lambda]\in \mathfrak{P}_{n,p}$ be irreducible. For any $[\mu]\neq [\lambda]$, $[s],[t]\in T([\mu])$ and $[u],[v]\in T([\lambda])$, we have
\begin{equation*}
    	C_p^{[\lambda]}([u],[v])C_p^{[\mu]}([s],[t])=0.
\end{equation*}
\end{lem}

The following proposition is a direct consequence of the two lemmas above:

\begin{prop}\label{impcm}
    Let $[\lambda]\in \mathfrak{P}_{n,p}$ be irreducible and $W([\lambda])$ be the cell module of $TL_{r,p,n}$ corresponding to $\lambda$. For any $[\mu]\neq [\lambda]$ and $[s],[t]\in T([\mu])$, we have
\begin{equation*}
    	C_p^{[\mu]}([s],[t])W([\lambda])=0.
\end{equation*}
\end{prop}

%\begin{prop} \label{bldecom}
%    Let $[\lambda]\in \mathfrak{P}_{n,p}$ be irreducible and $TL([\lambda])$ be the $TL_{r,p,n}$-bimodule defined before Lemma \ref{impr}. Then $TL([\lambda])$ is a block of $TL_{r,p,n}$. In other words, we have
 
 %   \begin{equation} \label{eqdecom}
 %       TL_{r,p,n}=TL_{r,p,n}^0\oplus (\oplus_{[\lambda]\in\mathfrak{P}_{n,p}^1}TL([\lambda]))
%    \end{equation}
 %   where $TL_{r,p,n}^0$ consists of the bimodules of reducible orbits in $\mathfrak{P}_{n,p}$. 
%\end{prop}

The following proposition is where the name ``irreducible multipartition" comes from. {Let $L([\lambda])=W([\lambda])/rad(\phi_{[\lambda]})$, where $\phi_{[\lambda]}$ 
is the standard bilinear form on a cell module defined in Definition 2.3 of \cite{Lehrer1996}.}
 We have:
\begin{prop} \label{835}
     Let $[\lambda]\in \mathfrak{P}_{n,p}$ be an irreducible multipartition, then $W([\lambda])=L([\lambda])$. In other words, the cell module $W([\lambda])$ is simple.
\end{prop}
\begin{proof}
  Let $[\mu]\in \mathfrak{P}_{n,p}$ be an equivalence class different from $[\lambda]$. By Proposition \ref{impcm}, $C_p^{[\mu]}([s],[t])W([\lambda])=0$ for all $[s],[t]\in T([\mu])$. On the other hand, since $L([\mu])=W([\mu])/rad(\phi_{[\mu]})$, {assuming that $L([\mu])\neq 0$,} for any $x\in L([\mu]) $, there exist $[s],[t]\in T([\mu])$ such that $C_p^{[\mu]}([s],[t])x\neq 0$. Therefore,we have
  \begin{equation*}
      [W([\lambda]):L([\mu])]=0
  \end{equation*}
  for all $[\mu]\neq [\lambda]$ {and $L([\mu])\neq 0$. By Proposition 3.6 in \cite{Lehrer1996}, $[W([\lambda]):L([\mu])]=1$,} so $W([\lambda])$ is simple and $W([\lambda])=L([\lambda])$.
\end{proof}
{
\begin{rem}\label{rem:cells}
1. Note that the discussion above shows that for all $[\lambda]\in\mathfrak{P}_{n,p}$, the form $\phi_{[\lambda]}\neq 0$, whence the corresponding simple module $L([\lambda])\neq 0$.

2. The converse of Proposition \ref{835} is not true. As a counterexample, the multipartitions with exactly one non-empty component are reducible by definition, 
but the corresponding cell modules are always simple since they are of rank 1.
\end{rem}
}
\subsection{The reducible cells and their decomposition numbers}
As shown in the last subsection, the cell modules corresponding to irreducible orbits are simple. To determine the decomposition numbers for all the cell modules, we concentrate on those corresponding to reducible orbits in this subsection. 

Let $\Lambda$ be the dominant weight we have chosen for the specialisation $H_n^{\Lambda}(q,\zeta)$ ($cf.$ \ref{domwt}) and $TL_{r,1,n}^{\Lambda}(q,\zeta)$ be the corresponding Temperley-Lieb algebra.  For $d=\frac{r}{p}$, let $\Lambda^0$ be the dominant weight of length $d$ such that $(\Lambda^0,\alpha_{(0,j)})=(\Lambda,\alpha_{(0,j)})$ for all $j\in J$. Let $\mathcal{R}_{n}^{\Lambda^0}$ be the cyclotomic KLR algebra corresponding to $\Lambda^0$ and $TL_{d,1,n}^{\Lambda^0}$ be the Temperley-Lieb quotient we define in Theorem \ref{dftl}. 
 The following lemma indicates that $TL_{d,1,n}^{\Lambda^0}$ is a quotient of $TL_{r,1,n}^{\Lambda}(q,\zeta)$:

\begin{lem}\label{lm835}
Let $TL_{r,1,n}^{\Lambda}(q,\zeta)$ be the Temperley-Lieb quotient of $H_n^{\Lambda}(q,\zeta)$ (cf. Theorem \ref{dftl}) and $TL_{d,1,n}^{\Lambda^0}$ be the Temperley-Lieb algebra as defined above. 
	Then $TL_{d,1,n}^{\Lambda^0}$ is a quotient of $TL_{r,1,n}^{\Lambda}(q,\zeta)$ by the two-sided ideal generated by all $e(\mathbf{i})=e(\mathbf{i}_1,\mathbf{i}_2,\dots,\mathbf{i}_n)$ where $\mathbf{i}_j$ are in the vertex set $K=\mathbb{Z}/p\mathbb{Z}\times\mathbb{Z}/e\mathbb{Z}$ and $(\Lambda^0,\alpha_{\mathbf{i}_1})=0$.
\end{lem}
\begin{proof}
   Let $\mathcal{R}_n^{\Lambda}$ be the KLR algebra isomorphic to $H_n^{\Lambda}(q,\zeta)$ (cf.Theorem 1.1 in \cite{Brundan_2009}).
	By comparing the generators and relations, we notice that $\mathcal{R}_n^{\Lambda^0}$ is a quotient of $\mathcal{R}_n^{\Lambda}$ by the two-sided ideal generated by all $e(\mathbf{i})$ such that $(\Lambda^0,\alpha_{\mathbf{i}_1})=0$ where $\mathbf{i}\in K^n$ and $\alpha_{\mathbf{i}_1}$ is the simple root {defined immediately after $(\ref{216})$}.
	
	By Theorem \ref{dftl}, $TL_{r,1,n}^{\Lambda}(q,\zeta)$ is a quotient of $\mathcal{R}_n^{\Lambda}$ by the two-sided ideal $\mathcal{J}_n(\Lambda)$ in (\ref{ieq}), and $TL_{d,1,n}^{\Lambda^0}$ is a quotient of $\mathcal{R}_n^{\Lambda^0}$ by the two-sided ideal $\mathcal{J}_n(\Lambda^0)$ in (\ref{ieq}). We only need to show that generators of $\mathcal{J}_n(\Lambda)$ in (\ref{con1}) and (\ref{con2}) are in the two-sided ideal of $\mathcal{R}_n^{\Lambda}$ generated by all $e(\mathbf{i})$ such that $(\Lambda^0,\alpha_{\mathbf{i}_1})=0$ and $\mathcal{J}_n(\Lambda^0)$. 
    {If $e(\mathbf{i})$ satisfies (\ref{con1}) with $\mathbf{i}_1$ not in form $(0,*)$, we have $(\Lambda^0,\alpha_{\mathbf{i}_1})=0$ because $\Lambda^0$ only consists of simple roots of the form $\Lambda_{(0,j)}$. If $e(\mathbf{i})$ satisfies (\ref{con1}) with $\mathbf{i}_1= (0,*)$, we have $(\Lambda^0,\alpha_{\mathbf{i}_1})=(\Lambda,\alpha_{\mathbf{i}_1})>0$ by the definition of $\Lambda^0$, thus $e(\mathbf{i}$ satisfies (\ref{con1}) for $\mathcal{J}_n(\Lambda^0)$. If  $e(\mathbf{i})$ satisfies (\ref{con2}), we also have $(\Lambda^0,\alpha_{\mathbf{i}_1})=(\Lambda,\alpha_{\mathbf{i}_1})>0$. }
    If $(\Lambda^0,\alpha_{\mathbf{i}_2})=0$ or $(\Lambda^0,\alpha_{\mathbf{i}_3})=0$, without losing generality, assume $(\Lambda^0,\alpha_{\mathbf{i}_2})=0$.  In this case, we use the fact {that} $e(\mathbf{i}_1,\mathbf{i}_2,\mathbf{i}_3,\mathbf{i}')=\psi_1e(\mathbf{i}_2,\mathbf{i}_1,\mathbf{i}_3,\mathbf{i}')\psi_1$ {(cf. \ref{sym})}, to show that $e(\mathbf{i}_1,\mathbf{i}_2,\mathbf{i}_3,\mathbf{i}')$ is in the two-sided ideal generated by all $e(\mathbf{i})$ such that $(\Lambda^0,\alpha_{\mathbf{i}_1})=0$.  
	
	Therefore, the ideal of $\mathcal{R}_n^{\Lambda}$ corresponding to $TL_{r,1,n}^{\Lambda}(q,\zeta)$ is in that corresponding to $TL_{d,1,n}^{\Lambda^0}$. So $TL_{d,1,n}^{\Lambda^0}$ is a quotient of $TL_{r,1,n}^{\Lambda}(q,\zeta)$. Moreover, we have:
	\begin{equation}
	    \begin{aligned}
	           TL_{d,1,n}^{\Lambda^0}&=\mathcal{R}_n^{\Lambda^0}/\mathcal{J}_n(\Lambda^0)\\
	           &=\mathcal{R}_n^{\Lambda}/\langle e(\mathbf{i}),\mathcal{J}_n(\Lambda^0)\rangle\\
	           &=\mathcal{R}_n^{\Lambda}/\langle e(\mathbf{i}),\mathcal{J}_n(\Lambda)\rangle\\
	           &=TL_{r,1,n}^{\Lambda}(q,\zeta)/\langle e(\mathbf{i})\rangle,
	    \end{aligned}
	\end{equation}
	where $\mathbf{i}$ runs over all $\mathbf{i}\in K^n$ such that $(\Lambda^0,\alpha_{\mathbf{i}_1})=0$.
\end{proof}
Let $f:TL_{r,1,n}^{\Lambda}(q,\zeta)\to TL_{d,1,n}^{\Lambda^0}$ be the natural quotient map. We have
	\begin{equation}
	    \begin{aligned}
	           f(y_k)&=y_k, \textit{ for } 1\leq k\leq n;\\
	           f(\psi_l)&=\psi_l, \textit{ for } 1\leq l \leq n-1;\\
	           f(e_{\lambda})&=
	           \begin{cases}
	               e_{\lambda}, \textit{ if } \lambda\in \mathfrak{P}_n^1;\\
	               0, \textit{ if } \lambda\in \mathfrak{P}_n- \mathfrak{P}_n^1,
	           \end{cases}
	    \end{aligned}
	\end{equation}
	where $\mathfrak{P}_n$ is the set of 3D multipartitions of $n$ with at most two non-empty components, and $\mathfrak{P}_n^1$ is the subset consisting of those 3D multipartitions with all the non-empty components in the first column of the floor table.
Since $TL_{r,p,n}^{\Lambda}$ is a subalgebra of $TL_{r,1,n}^{\Lambda}(q,\zeta)$, $f(TL_{r,p,n}^{\Lambda})$ is a subalgebra of $TL_{d,1,n}^{\Lambda^0}$. We show the map $f$ is surjective. By Theorem \ref{cb419}, $TL_{d,1,n}^{\Lambda^0}$ has a cellular basis $\{C_{s,t}^{\lambda}\}$ corresponding to the cell datum $(\mathfrak{B}_n^{(d)},Std,C,deg)$, {where $\mathfrak{B}_n^{(d)}$ is the poset of one-column $d$-partitions of $n$ with at most two non-empty components}. For any $\lambda_0\in \mathfrak{B}_n^{(d)}$ and $s_0,t_0 \in Std(\lambda_0)$, let $j_1,j_2,\dots,j_n\in\mathbb{Z}/e\mathbb{Z}$ be such that $e_{\lambda_0}=e((0,j_1),(0,j_2),\dots,(0,j_n))$. We have $C_p^{[\lambda_0]}([s_0],[t_0])\in TL_{r,p,n}^{\Lambda}$ and 
\begin{equation}\label{eqquomap}
\begin{aligned}
    f(C_p^{[\lambda_0]}([s_0],[t_0]))&=f(\psi_{d([s_0])}^*(\sum_{\lambda'\equiv\lambda_0}e_{\lambda'})\psi_{d([t_0])})\\
    &=\psi_{d([s_0])}^*f(\sum_{\lambda'\equiv\lambda_0}e_{\lambda'})\psi_{d([t_0])}\\
    &=\psi_{d([s_0])}^*f(\sum_{i=0}^{p-1}e((i,j_1),(i,j_2),\dots,(i,j_n)))\psi_{d([t_0])}\\
                              &=\psi_{d([s_0])}^*e((0,j_1),(0,j_2),\dots,(0,j_n))\psi_{d([t_0])}\\
                              &=C_{s_0,t_0}^{\lambda_0}.
\end{aligned}
\end{equation}
Therefore, the restriction of the natural map $f$ on $TL_{r,p,n}^{\Lambda}$ is surjective, that is  $f(TL_{r,p,n}^{\Lambda})=TL_{d,1,n}^{\Lambda^0}$.

The next lemma shows that the kernel of this map is the subspace spanned by the elements corresponding to irreducible orbits in the cellular basis in $(\ref{cbrpn})$. 

\begin{lem}\label{836}
Let $TL_{r,p,n}^1$ be the subspace spanned by all elements $C_p^{[\lambda]}([s],[t])$ in the cellular basis of $TL_{r,p,n}^{\Lambda}$ in (\ref{cbrpn}), where $[\lambda]$ is irreducible. Then $TL_{r,p,n}^1$ is the kernel of $f|_{TL_{r,p,n}^{\Lambda}}$ where $f$ is the natural quotient map defined above. Moreover, $TL_{r,p,n}^1$ is an ideal of $TL_{r,p,n}^{\Lambda}$ and
\begin{equation}
    TL_{d,1,n}^{\Lambda^0}\cong TL_{r,p,n}^{\Lambda}/TL_{r,p,n}^1
\end{equation}
\end{lem}
\begin{proof}
We first show that $f(C_p^{[\lambda]}([u],[v]))=0$ if $[\lambda]$ is irreducible. Using the notation introduced before Definition \ref{2PO}, let $\lambda=\lambda_{(i_1,l_1),(i_2,l_2)}^{[a]}$. As $\lambda$ is irreducible, we have $i_1\neq i_2$ and $e_{\lambda}=e((i_1,j_{l_1}),(i_2,j_{l_2}),\mathbf{k})$ with $\mathbf{k}\in K^{n-2}$. We have
\begin{equation}
\begin{aligned}
       e_{[\lambda]}&=\sum_{\lambda'\equiv\lambda}e_{\lambda}  \\
       &=\sum_{i=0}^{p-1} e((i_1-i,j_{l_1}),(i_2-i,j_{l_2}),\sigma_1^i(\mathbf{k})),
\end{aligned}
\end{equation}
where $\sigma_1$ is the map defined in (\ref{8.14}). For $i\neq i_1$, $e((i_1-i,j_{l_1}),(i_2-i,j_{l_2}),\sigma_1^i(\mathbf{k}))$ is in the ideal described in Lemma \ref{lm835} {because $\Lambda^0$ 
consists precisely of the simple roots of the form $\Lambda_{(0,j)}$, so that $(\Lambda^0,\alpha_{(i_1-i,j_{l_1})})=0$}. If $i=i_1$, we have $i_2-i\neq 0$ and $e((i_1-i,j_{l_1}),(i_2-i,j_{l_2}),\sigma_1^i(\mathbf{k}))=\psi_1 e((i_2-i,j_{l_2}),(i_1-i,j_{l_1}),\sigma_1^i(\mathbf{k}))\psi_1$. So $e((i_1-i,j_{l_1}),(i_2-i,j_{l_2}),\sigma_1^i(\mathbf{k}))$ is in that ideal. Then we have $e_{[\lambda]}\in ker(f)$, thus $C_p^{[\lambda]}([u],[v])\in ker(f)$ for any $u,v\in Std(\lambda)$.
Therefore, we have $TL_{r,p,n}^1\subseteq ker(f) $.

 On the other hand, for any reducible orbit $[\mu]\in \mathfrak{P}_{n,p}$ and $[s],[t]\in T_p([\mu])$(cf.\ref{t_p}), equation (\ref{eqquomap}) shows that $C_p^{[\mu]}([s],[t])$ is not in the kernel of $f$. So $TL_{r,p,n}^1= ker(f|_{TL_{r,p,n}^{\Lambda}}) $. Since $f|_{TL_{r,p,n}^{\Lambda}}$ is a surjective algebra homomorphism, $TL_{r,p,n}^1$ is an ideal of $TL_{r,p,n}^{\Lambda}$ and
\begin{equation*}
    TL_{d,1,n}^{\Lambda^0}\cong TL_{r,p,n}^{\Lambda}/TL_{r,p,n}^1.
\end{equation*}\end{proof}

The next lemma is an immediate consequence of the definition of cellular algebras and shows that $ TL_{d,1,n}^{\Lambda^0}$ inherits a cellular structure from $TL_{r,p,n}^{\Lambda}$ as a quotient.

\begin{lem}
Let the algebras  $TL_{r,p,n}^{\Lambda}$ and $TL_{r,p,n}^1$ be as defined above. Then $TL_{r,p,n}^{\Lambda}/TL_{r,p,n}^1$ is a graded cellular algebra with cellular datum $(\mathfrak{P}_{n,p}^0,T_p,\mathfrak{C}_p,deg_p)$, where $\mathfrak{P}_{n,p}^0$ is the subset of $\mathfrak{P}_{n,p}$ consisting of all the reducible orbits, $T_p,deg_p$ are the same as the ones in Theorem \ref{csrpn} and $\mathfrak{C}_p^{[\mu]}([s],[t])=C_p^{[\mu]}([s],[t])+TL_{r,p,n}^1$ for all $[\mu]\in \mathfrak{P}_{n,p}^0$ and $[s],[t]\in T_p( [\mu])$.
\end{lem}

For $[\mu]\in \mathfrak{P}_{n,p}^0$ and $[\lambda]\in \mathfrak{P}_{n,p}^1$, Lemma \ref{832} implies
\begin{equation}
    	C_p^{[\lambda]}([u],[v])W([\mu])=0.
\end{equation}
Thus the kernel of the map $f$ above acts trivially on $W([\mu])$, whence the  action of $TL_{r,p,n}$ on the cell module $W([\mu])$ is exactly the same as that of $TL_{d,1,n}^{\Lambda^0}$. Hence we need only to determine the decomposition numbers for the cell modules $W([\mu])$, regarded as $TL_{d,1,n}^{\Lambda^0}$-modules.

By Theorem \ref{cb419}, the generalised Temperley-Lieb algebra $TL_{d,1,n}^{\Lambda^0}$ has another cell datum $(\mathfrak{B}_n^{(d)}, Std, C, deg)$ where $\mathfrak{B}_n^{(d)}$ consists of the 3D multipartitions with at most two non-empty components and both of them are in the first column of the floor table. The decomposition numbers corresponding to this cell datum are given in the last chapter. To show that these two cell data are the same, we first choose a special representative in each reducible orbit.
\begin{defn} \label{orep}
    For any reducible orbit $[\mu]\in \mathfrak{P}_{n,p}^0$, define the original representative $\mu_0$ in $[\mu]$ as the multipartition such that all the non-empty components are in the first column of the floor table.
\end{defn}
 Since the non-empty components in a reducible multipartition are in the same column, and the map $\sigma_{\mathfrak{P}_n}$ moves multipartitions one column ahead, there exists a unique original representative $\mu_0$ in each reducible orbit $[\mu]$. For $[t]\in\mathcal{T}_p([\mu])$, denote by $t_0\in [t]$ the standard tableau of shape $\mu_0$. Let $f:TL_{r,p,n}^{\Lambda}\to TL_{d,1,n}^{\Lambda^0}$ be the natural quotient map.Then we have
 \begin{align*}
     f(C_p^{[\mu]}([s],[t]))&=f(\psi_{d([s])}^*(\sum_{\mu'\equiv\mu_0}e_{\mu'})\psi_{d([t])})\\
    &=\psi_{d([s])}^*f(\sum_{i=0}^{p-1}e_{\mu_i})\psi_{d([t])}\\
    &=\psi_{d(s_0)}^*f(e_{\mu_0})\psi_{d(t_0)}\\
  &=C^{\mu_0}(s_0,t_0).
\end{align*}
Therefore, we have 
\begin{equation}
    \mathfrak{C}_p^{[\mu]}([s],[t])=C^{\mu_0}(s_0,t_0),
\end{equation}
for all  $[\mu]\in \mathfrak{P}_{n,p}^0$ and $[s],[t]\in T_p([\mu])$. As a direct consequence, we have
\begin{lem}
For any reducible orbit $[\mu]$, let $ W([\mu])$(resp. $L([\mu])$) be the cell ( resp. simple) module of $TL_{d,1,n}^{\Lambda^0}$ with respect to the cell datum $(\mathfrak{P}_{n,p}^0,T_p,\mathfrak{C}_p,deg_p)$. Denote by $\mu_0$ the original representative in $[\mu]$. Let $W(\mu_0)$ ($L(\mu_0)$) be the cell (resp. simple) module with respect to the cell datum $(\mathfrak{B}_n^{(d)}, Std, C, deg)$. Then we have
\begin{equation*}
        W([\mu])\cong W(\mu_0);\qquad
        L([\mu])\cong L(\mu_0).
\end{equation*}
\end{lem}
Moreover, let $[\lambda],[\mu]$ be two reducible orbits. As $TL_{d,1,n}^{\Lambda^0}$-modules, we have
\begin{equation}\label{finaltool}
     [W([\lambda]):L([\mu])]=[ W(\lambda_0):L(\mu_0)].
\end{equation}

Since $TL_{r,p,n}^1$ acts trivially on $W([\lambda])$ where $[\lambda]$ is reducible, the left-hand side equals the decomposition numbers for the corresponding cell modules of $TL_{r,p,n}$. The right-hand side can be obtained from Theorem 6.23 in \cite{LL22}.
We are now in a position to describe the decomposition numbers of $TL_{r,p,n}$.

\begin{thm} \label{finaldcm}
       Let $TL_{r,p,n}$ be the Temperley-Lieb algebra for the imprimitive complex reflection group $G(r,p,n)$ as defined in Definition \ref{dfrpn}. According to Theorem \ref{csrpn}, $TL_{r,p,n}$ is a cellular algebra with respect to the poset $\mathfrak{P}_{n,p}$ in $(\ref{poset831})$.
       For an orbit of multipartitions $[\lambda]\in\mathfrak{P}_{n,p}$, let $W([\lambda])$ and $L([\lambda])$ be the cell and simple modules corresponding to $[\lambda]$, respectively.
       
       If $\lambda$ is irreducible (cf. Definition \ref{irrr}), then
       \begin{equation}
           W([\lambda])=L([\lambda]).
       \end{equation}
       
       If $\lambda$ is reducible, $W([\lambda])$ has no decomposition factors of the form $L([\mu])$ where $\mu$ is irreducible. 
       
       For any reducible orbit $[\mu]\in \mathfrak{P}_{n,p}$, let $\lambda_0$ and $\mu_0$ be the original representatives as defined in Definition \ref{orep}. Then 
		\begin{equation}
		[W([\lambda]),L([\mu])]=
		\begin{cases}
			1 &\text{ if } \lambda_0 \unlhd\mu_0 \text{ and there exists } t_0\in Std(\lambda_0) \text{ such that } e(t_0)=e_{\mu_0};\\
			0 &\text{ otherwise,}
		\end{cases}
	\end{equation}
		where $e(t_0)$ and $e_{\mu_0}$ are the KLR generators.
\end{thm}
\begin{proof}
If $\lambda$ is irreducible, $W([\lambda])$ is simple by Proposition \ref{835}, so we have
\begin{equation*}
    W([\lambda])=L([\lambda]).
\end{equation*}

If $\lambda$ is reducible but $\mu$ is irreducible, Lemma \ref{832} implies that
 \begin{equation*}
    	C_p^{[\lambda]}([s],[t])W([\mu])=0,
\end{equation*}
 for all $[s],[t]\in T([\lambda])$. On the other hand, for any non-zero element $x\in L([\lambda]) $, there exist $[s],[t]\in T([\lambda])$ such that $C_p^{[\lambda]}([s],[t])x\neq 0$. Therefore, we have
  \begin{equation*}
      [W([\mu]):L([\lambda])]=0.
  \end{equation*}

   If both $\lambda$ and $\mu$ are reducible, then $(\ref{finaltool})$ applies.
 Hence by Theorem 6.23 in \cite{LL22},
 		\begin{align*}
 		 [W([\lambda]),L([\mu])]&=[ W(\lambda_0):L(\mu_0)]\\
 		 &=
		\begin{cases}
			1 &\text{ if } \lambda_0 \unlhd\mu_0 \text{ and there exists } t_0\in Std(\lambda_0) \text{ such that } e(t_0)=e_{\mu_0};\\
			0 &\text{ otherwise.}
		\end{cases}   
 		\end{align*}
\end{proof}

\section{Further developments: speculations on potential categorification and diagrammatics}

The original Temperley-Lieb algebras $TL_n(q)$ and the Temperley-Lieb algebras of ``type $B$'', $TL_n(Q,q)$, fit into a category, where they are the algebras of endomorphisms of the
objects. These categories may be shown to be equivalent to certain categories of representations of  quantum $\mathfrak{sl}_2$ (see \cite{IoharaLehrerZhang2021}). We briefly discuss here how 
our theory of generalised Temperley-Lieb algebras might be further developed in this direction.

Similarly to the categorification of $TLB_n(Q,q)$ in \cite{GRAHAM2003479}, we may define a Temperley-Lieb category of type $G(r,p,n)$, 
written $\mathbb{TL}_{r,p}$, where the objects are non-negative integers $\mathbb{N}$. Instead of the marked diagrams, we use linear combinations of standard tableau pairs as morphisms. 

Let $m, n$ be a pair of positive integers of the same parity. Using the notation above Definition \ref{irrr}, denote by $\lambda_{(i_1,l_1),(i_2,l_2)}^{[a]}$ a 3D-multipartition of $m$ and by $\mu_{(i_1,l_1),(i_2,l_2)}^{[a]}$ a 3D-multipartition of $n$. Let $t$ and $s$ be any standard tableaux of shape $\lambda$ and $\mu$, respectively. Then the pair $(t, s): m\mapsto n$ forms a morphism from $m$ to $n$ in $\mathbb{TL}_{r, p}$. Composition of two morphisms $(t, s): m\mapsto n$ and $(u, v): n\mapsto l$ could be defined as $(t, v): m\mapsto l$  multiplied by the bilinear form $\phi_{\mu}(s, u)$ in $TL_{r, p, n}$.

 It is an open question whether this generalised Temperley-Lieb category can be realised as a category of some representations, or indeed whether there are interesting
 functors from $\mathbb{TL}_{r,p}$ to categories of representations.

The categorification above generalises  that of $TLB_n(Q,q)$ in \cite{GRAHAM2003479}. When $r=2$ and $p=1$, the 3D-multipartition turns into a bipartition. 
To realise the objective of representing $\mathbb{TL}_{r, p}$ diagramatically, we offer the following hints, which we hope to implement in a future work. 

We can construct a bijective map $\psi_b$ between the standard tableaux of bipartitions of $n$ and the monic diagrams of type B, which is a collection of bijections $\psi_{b,\lambda}$ for all bipartitions $\lambda$ of $n$.
To construct the bijective map $\psi_{b, \lambda}$, we need to define the sign $(\pm)$ of a point in a marked monic diagram. A horizontal string is called negative if it is marked or covered by a marked string. 
Other horizontal strings are positive. A point is called negative if it is the left end of a negative string or a right end of a positive one. 
Otherwise, the point is called positive.  Figure-$\ref{fig-NgB}$ depicts an example of positive and negative points in a monic diagram of type B. 
{ Given a standard tableau $t$ of shape $\lambda$ with $j_1< j_2< \dots<j_i$ in the first component and $k_1<k_2<\dots <k_l$ in the second component, define the image $\psi_{b, \lambda}(t)$
to be  a monic diagram $D$, defined as follows:
\begin{itemize}
    \item If $i\geq l$, $D:n\mapsto n-2l$ is the monic diagram with $j_1< j_2< \dots<j_i$ positive and $k_1<k_2<\dots <k_l$ negative on the top line.
    \item If $i< l$, $D:n\mapsto n-2i$ is the monic diagram with $j_1< j_2< \dots<j_i$ negative and $k_1<k_2<\dots <k_l$ positive on the top line and add a blob on the left-most through string.
\end{itemize}} 

\begin{figure}[h]
	\begin{center}
		\begin{tikzpicture}[scale=1]
			
			%% draw filled dots at required points at bottom
			\foreach \x in {1,2,3,4,5,6,7}
			\filldraw(\x,2) circle (0.05cm);
			%% draw filled dots at required points on top
			
			\filldraw(4,0) circle (0.05cm);
			\filldraw(2.5,1.4) circle (0.1cm);
			%% draw straight lines where required
			%\foreach \x in {1,4,6,9,11,14}
			%\draw (\x,3)--(\x,0);
			\draw (4,0)--(5,2);

			% draw fundamental rectangle

			\draw [dashed] (0,0)--(8,0); \draw [dashed] (0,2)--(8,2);
			\draw [dashed] (0,0)--(0,2);\draw [dashed] (8,0)--(8,2);
			
			%% label the points in top row
			%%\draw node[left] at (1,2){1};\draw node[above] at (2,2){2};\draw node[above] at (3,2){3};
			\draw node[below] at (4,0){$1'$};
			
			%% label the points in bottom row
			\draw node[above] at (1,2){1(-)};\draw node[above] at (2,2){2(-)};\draw node[above] at (3,2){3(+)};
			\draw node[above] at (4,2){4(+)};\draw node[above] at (5,2){5(+)};\draw node[above] at (6,2){6(+)};\draw node[above] at (7,2){7(-)};

			%% draw curved lines where required, top
			%\draw(6.5,2) .. controls (6.7,1.4) and (7.3,1.4) .. (7.5,2);

			%% draw curved lines where required, bottom
			%\draw(5,0) .. controls (6,0.8) and (9,0.8) .. (10,0);
			%\draw(7,0) .. controls (7.5,1) and (9.5,1) .. (10,0);
			%\draw(6,0) .. controls (7,1.5) and (10,1.5) .. (11,0);
			\draw(1,2) .. controls (2,1.2) and (3,1.2) .. (4,2);
			
			\draw(2,2) .. controls (2.2,1.7) and (2.8,1.7) .. (3,2);
			\draw(6,2) .. controls (6.2,1.7) and (6.8,1.7) .. (7,2);
			%% draw ... between straight lines
			%\draw(4.5,0) node {\large{$\cdots$}};
			%\draw(4.5,2) node {\large{$\cdots$}};
			%\draw(12.5,1.5) node {\large{$\cdots$}};
			
			%% label the diagrams
			%\draw node[below] at (3,2.8){$S$};\draw node[below] at (10,2.8){$T$};\draw node[below] at (6.5,-0.2){$D_{S,T}^3$};
			
		\end{tikzpicture}
		%\centerline{Figure 2}
	\end{center}
	\caption{Positive and negative points of Type B} \label{fig-NgB}
\end{figure}

{The existence and uniqueness of $\psi_{b, \lambda}(t)$ may be proved by induction on $n$.}

Given bipartitions $\lambda_{(1,1),(2,2)}^{[a]}$ of $m$ and $\mu_{(1,1),(2,2)}^{[a]}$ of $n$, we obtain two monic diagrams $C: m\mapsto |a|$ and $D: n\mapsto |a|$. 
By suitably concatenating these two diagrams, we obtain a diagrammatic morphism in the category $\mathbb{TB}$ which is defined in  \cite{GRAHAM2003479}.

This observation indicates that it may be possible to find a diagrammatic description for $TL_{r, p, n}$ through the tableaux and the cellular basis in Theorem \ref{csrpn}. One outstanding obstacle 
to completing this program is that the bilinear form $\phi_{\lambda}$, which is needed to replace loops in the concatenation, depends on the multipartition $\lambda$ and a formula
for its value is not yet known. 
\bibliographystyle{hsiam}
\bibliography{refs2}

\begin{thebibliography}{10}

\bibitem{ARIKI1995164}
{\sc S.~Ariki}, {\em Representation theory of a {H}ecke algebra of ${G}(r, p, n)$}, Journal of Algebra, 177 (1995), pp.~164--185.

\bibitem{broue1994complex}
{\sc M.~Broué, G.~Malle, and R.~Rouquier}, {\em Complex reflection groups, braid groups, {H}ecke algebras}, J. Reine Angew. Math., 1998 (1998), pp.~127--190.

\bibitem{Brundan_2009}
{\sc J.~Brundan and A.~Kleshchev}, {\em Blocks of cyclotomic {H}ecke algebras and {K}hovanov-{L}auda algebras}, Inventiones mathematicae, 178 (2009), p.~451–484.

\bibitem{BrundanKleshchev_2009}
\leavevmode\vrule height 2pt depth -1.6pt width 23pt, {\em Graded decomposition numbers for cyclotomic {H}ecke algebras}, Advances in Mathematics, 222 (2009), p.~1883–1942.

\bibitem{Lehrer1996}
{\sc J.~Graham and G.~Lehrer}, {\em Cellular algebras.}, Inventiones mathematicae, 123 (1996), pp.~1--34.

\bibitem{GRAHAM2003479}
\leavevmode\vrule height 2pt depth -1.6pt width 23pt, {\em Diagram algebras, hecke algebras and decomposition numbers at roots of unity}, Annales Scientifiques de l’École Normale Supérieure, 36 (2003), pp.~479--524.

\bibitem{GROSSMAN2007477}
{\sc P.~Grossman}, {\em Forked {Temperley–Lieb} algebras and intermediate subfactors}, Journal of Functional Analysis, 247 (2007), pp.~477--491.

\bibitem{hu2021skew}
{\sc J.~Hu, A.~Mathas, and S.~Rostam}, {\em Skew cellularity of the {H}ecke algebras of type {$G(\ell,p,n)$}}, arXiv e-prints,  (2021), 2106.11459.

\bibitem{IoharaLehrerZhang2021}
{\sc K.~Iohara, G.~Lehrer, and R.~Zhang}, {\em Equivalence of a tangle category and a category of infinite dimensional $\mathrm {U}_q(\mathfrak {sl}_2)$-modules}, Representation Theory of the American Mathematical Society, 25 (2021), pp.~265--299.

\bibitem{khovanov2008diagrammatic}
{\sc M.~Khovanov and A.~D. Lauda}, {\em A diagrammatic approach to categorification of quantum groups i}, Representation Theory, 13 (2009), pp.~309--347.

\bibitem{LL22}
{\sc G.~Lehrer and M.~Lyu}, {\em Generalised {Temperley-Lieb} algebras of type {$G(r,1,n)$}}, arXiv e-prints,  (2022), 2209.00186.

\bibitem{Li2021}
{\sc Y.~Li and X.~Shi}, {\em Semi-simplicity of {Temperley-Lieb} algebras of type {D}}, arXiv e-prints,  (2021), 2003.06805.

\bibitem{LOBOSMATURANA2020106277}
{\sc D.~L. Maturana and S.~Ryom-Hansen}, {\em Graded cellular basis and {Jucys-Murphy} elements for generalized blob algebras}, Journal of Pure and Applied Algebra, 224 (2020).

\bibitem{murphy1995representations}
{\sc G.~Murphy}, {\em The representations of hecke algebras of type $a_n$}, Journal of Algebra, 173 (1995), pp.~97--121.

\bibitem{Rostam_2018}
{\sc S.~Rostam}, {\em Cyclotomic quiver hecke algebras and {H}ecke algebra of {$G(r,p,n)$}}, Transactions of the American Mathematical Society, 371 (2018), p.~3877–3916.

\bibitem{rouquier20082kacmoody}
{\sc R.~Rouquier}, {\em 2-kac-moody algebras}, arXiv: Representation Theory,  (2008).

\bibitem{Sun2009}
{\sc J.~Sun}, {\em Cyclotomic {Temperley–Lieb} algebra of type {D} and its representation theory}, Communications in Algebra, 38 (2009), pp.~94--118.

\bibitem{TL1971}
{\sc H.~N.~V. Temperley and E.~H. Lieb}, {\em Relations between the 'percolation' and 'colouring' problem and other graph-theoretical problems associated with regular planar lattices: Some exact results for the 'percolation' problem}, Proceedings of the Royal Society of London. Series A, Mathematical and Physical Sciences, 322 (1971), pp.~251--280.

\end{thebibliography}
\end{document}